\documentclass[a4paper,leqno]{article}
\usepackage[a4paper,centering,vscale=0.77,hscale=0.77]{geometry}\usepackage{amsmath,amsthm,amssymb,bbm,bm}
\usepackage{amsmath,amsthm,amssymb,bbm,bm}
\usepackage{xcolor}
\usepackage{verbatim}
\usepackage{todonotes}
\usepackage{mathrsfs}

\usepackage{thmtools}
\usepackage[pdfdisplaydoctitle,colorlinks,breaklinks,urlcolor=blue,linkcolor=blue,citecolor=blue]{hyperref}
\usepackage{enumitem}
\usepackage{cite}

\newtheorem{theorem}{Theorem}[section]
\newtheorem{lemma}[theorem]{Lemma}
\newtheorem{proposition}[theorem]{Proposition}
\newtheorem{corollary}[theorem]{Corollary}
\newtheorem{definition}[theorem]{Definition}
\newtheorem{remark}[theorem]{Remark}

\newtheorem{notation}[theorem]{Notation}

\newcommand{\df }{\mathrm{d}}
\newcommand{\im }{\mathrm{i}}

\newcommand{\Prob}{\mathbb{P}}
\newcommand{\Filtration}{\mathbb{F}}

\newcommand{\F}{\mathscr{F}}

\def\cB{\mathscr{B}}

\numberwithin{equation}{section}

\begin{document}

\title{
Global well posedness and ergodic results in regular Sobolev spaces  for the nonlinear Schr\"odinger equation with multiplicative noise and arbitrary power of the nonlinearity} 
\author{Zdzis{\l}aw Brze{\'z}niak
\thanks{Department of Mathematics, University of York,
			Heslington, York, YO105DD, UK
		E-mail: \texttt{zdzislaw.brzezniak@york.ac.uk}}
, Benedetta Ferrario
\thanks{ Dipartimento di Scienze Economiche e Aziendali, Universit\`a di Pavia, 27100 Pavia, Italy.
E-mail: \texttt{benedetta.ferrario@unipv.it}}
, Mario Maurelli
\thanks{
Dipartimento di Matematica, Universit\`a di Pisa, Largo Bruno Pontecorvo 5, 56127 Pisa, Italy.
E-mail: \texttt{mario.maurelli@unipi.it}}
, Margherita Zanella
\thanks{Dipartimento di Matematica, Politecnico di Milano,
Via E.~Bonardi 9, 20133 Milano, Italy.
E-mail: \texttt{margherita.zanella@polimi.it}}
}


\maketitle

\begin{abstract} We consider the nonlinear Schr\"odinger equation on the $d$-dimensional torus $\mathbb T^d$, with the nonlinearity of polynomial type $|u|^{2\sigma}u$. For any $\sigma \in \mathbb N$ and $s>\frac d2$ we prove that adding to this equation a  suitable  stochastic forcing term  there exists a unique global solution for {\em any} initial data in $H^s(\mathbb T^d)$.
The effect of the noise is to prevent blow-up in finite time, differently from the deterministic setting. Moreover we prove existence of invariant measures and their uniqueness under more restrictive assumptions on the noise term.
\end{abstract}


\bigskip

\tableofcontents



\section{Introduction}
The nonlinear Schr\"odinger equation is  one of the basic models
for nonlinear waves and  has many physical applications,  e.g. to nonlinear optics, plasma
physics and quantum field theory, see e.g. \cite{SulSul,Turitsyn_2012} and references therein.
\\
In this article, we deal with the  nonlinear Schr\"odinger equation (NLS)  on the $d$-dimensional torus $\mathbb{T}^d:=(\mathbb{R}/2\pi\mathbb{Z})^d$, for $d\ge 1$, with  stochastic forcing term
\begin{equation}
\label{NLS}
\begin{cases}
{\rm d} u(t,x)+\left[ \im \Delta  u(t,x)+\im  \alpha |u(t,x)|^{2\sigma} u(t,x)  \right] \,{\rm d}t 
= \phi(u(t,x))  \,{\rm d} W(t), \quad (t, x) \in (0, \infty) \times \mathbb{T}^d,
\\
u(0,x)=u^{0}(x), \quad x \in \mathbb{T}^d.
\end{cases}
\end{equation}
Here  
  $\im=\sqrt{-1} $,  $\sigma>0$ and $W=(W(t):t\geq 0)$ is a classical real-valued Wiener process.  The equation is called  focusing when $\alpha=+1$ and  defocusing when $\alpha=-1$. In the present paper we will study a more general case when 
  $\alpha \in \mathbb C$; therefore  one can view  equation \eqref{NLS}  as 
  the stochastic version of the equation considered by Kato in \cite{Kato}, see also \cite[Section 4.4]{Cazenave}.

Our first main result in this paper is  a kind  of regularization by noise for the deterministic NLS equation, in the following sense. 
We find sufficient conditions on the diffusion coefficient $\phi$ for the  global-in-time existence of  the solutions to the stochastic NLS equation \eqref{NLS}, whereas a similar result does not hold for the corresponding deterministic problem (i.e. equation \eqref{NLS} with $\phi=0$).
To be more precise, for every $\sigma \in \mathbb{N}$ we find a sufficient condition on $\phi$, such that if  $\alpha \in \mathbb{C}$, $s>\frac{d}2$ and  $u^0 \in H^s(\mathbb{T}^d)$, then the NLS equation \eqref{NLS} has a global solution. 
In particular, the global existence holds in both focusing and defocusing cases. Our proof relies on a tightness method based on the choice of a suitable Lyapunov function. 
Secondly, we study the long time behavior of the solution process. We prove the existence of \emph{stationary} solutions,  again under suitable assumptions on $\phi$  by modifying the Lyapunov function. 
Moreover, under additional conditions of $\phi$, we prove that
all the solutions tend to zero, almost surely, and we establish some stability results for the zero solution. Hence,
$   \mu=\delta_0$ is the unique invariant measure.  

Let us review in more details the known results.

\subsection{Existence of global solutions}

From the mathematical point of view, the question of the (global or local)  existence and of the uniqueness of the solutions  has been widely studied; the starting model is the unforced NLS equation
\begin{equation}
\label{NLS_det}
\begin{cases}
\partial_t u(t,x)+ \im \Delta  u(t,x)+\im  \alpha |u(t,x)|^{2\sigma} u(t,x) 
=0, &
\\
u(0,x)=u^{0}(x).&
\end{cases}
\end{equation}
Different technical issues have to be  faced, depending on the spatial domain (its shape and its dimension), 
the power of the nonlinearity $\sigma$ and the parameter $\alpha=\pm 1$. 

The first global-in-time existence results  for the \textbf{deterministic} NLS equation \eqref{NLS_det} have been obtained  
in the full space $\mathbb R^d$
for $\sigma<\sigma_{cr}$, where the critical value $\sigma_{cr}$ depends on  $\alpha$, the spatial dimension $d$ 
and the state space  in which the dynamics leaves. Then big efforts have been done to deal with the critical  case 
$\sigma=\sigma_{cr}$ and the supercritical  case $\sigma>\sigma_{cr}$. 

A lot of papers treat what is called the energy subcritical and critical cases, i.e. $0\le \sigma \le \frac 2{(d-2)_+}$ in the defocusing case ($\alpha=-1 $) and $0\le \sigma \le \frac 2d$ in the focusing case ($\alpha=1$). It is well known that solutions of the NLS equation \eqref{NLS_det} satisfy various conservation laws. In particular,  for any smooth enough solution the mass 
\[
\mathcal{M}(u):=\|u\|_{L^2}^2
\]
and the energy 
\[
\mathcal{E}(u):=\frac 12 \|\nabla u\|_{L^2}^2 -\frac \alpha{2+2\sigma} \|u\|_{L^{2+2\sigma}}^{2+2\sigma}
\]
are conserved in time. Thus, it seems natural to study the problem in $H^1(\mathbb{R}^d)$, the so-called energy space. In fact, both in the focusing and defocusing case, the  conserved quantities can be used to get bounds in $H^1(\mathbb{R}^d)$ and show the existence of global solutions for \eqref{NLS}; see e.g. \cite{Cazenave}. The $H^1$-theory has been limited to sub-critical and critical powers; for larger values of $\sigma$ (supercritical power), the nonlinear term is difficult to handle.
In the supercritical case the Cauchy problem for \eqref{NLS_det} considered in $\mathbb{R}^d$ is solved only locally in time or globally for small data on some Sobolev spaces (see \cite[Section 6.2]{Cazenave} and \cite{CW}). 
Moreover, cases in which there is blow-up in finite time are known. For instance,  there  is blow-up in finite time  
to \eqref{NLS_det} in the focusing supercritical case, i.e.  
$\frac 2d\le \sigma \le\frac{2}{(d-2)_+}$ and $\alpha=1$,
provided, e.g. the energy of the initial data is negative  (see e.g. \cite[Chapter 6.5]{Cazenave}). 
In \cite{Merle_2022} for the defocusing NLS equation there are examples of   energy supercritical parameters $\sigma $ when $d\ge 5$ for which 
 there  is blow-up in finite time  if the initial data $C^\infty$  are well localized spherically symmetric  functions.
 
A challenging open problem is the global well-posedness of the NLS equation \eqref{NLS} in the more regular Sobolev space $H^s(\mathbb{R}^d)$, with $s>\frac d2$. 
The proof of the existence of a local solution is not an issue in $H^s(\mathbb{R}^d)$ since, via the Sobolev embedding $H^s(\mathbb{R}^d) \subset L^\infty(\mathbb{R}^d)$, one easily controls the nonlinear term (with an arbitrary large power); 
see e.g. \cite[Section 4.10]{Cazenave}.  
However the existence of a \emph{global} solution is harder, since a priori control in the $H^1$-norm of the solution no longer implies a priori control of the   $H^s$-norm, when $d \ge 2$. 
Hence, one cannot use the conservation of the mass and the energy to deduce the existence of  a global solution. 

For what concerns the NLS equation \eqref{NLS_det} on a $d$-dimensional torus the literature is much scarcer. The behaviour in the periodic setting 
is  strictly worse than in the case of the full space. 
The $H^1(\mathbb{T}^d)$-critical and subcritical cases are the more studied, starting from the work by Bourgain \cite{Bourgain}; see also \cite{BGT04, Herr+Tataru+Tzvetkov_2011} and the references therein. On the torus case as well,  sufficient conditions
for blow-up have been studied; see, e.g.,  \cite{OT} for the 
 the one dimensional quintic focusing NLS equation  
  and the therein references. 
  
The global well-posedness of the NLS equation \eqref{NLS_det} in the Sobolev space $H^s$, for $s>\frac d2$ remains a challenging problem also in the case of $\mathbb{T}^d$, but recently some progress in this direction has been done in the beautiful paper by Sy \cite{Sy}.
He considers the NLS equation on $\mathbb{T}^3$ with arbitrary large parameter $\sigma$ and constructs 
non-trivial  invariant probability measures $\mu$ supported on the Sobolev space $H^{s}(\mathbb{T}^3)$, $s\ge 2$.
Then, he shows that the equations are globally well-posed on the supports of these measures. As far as we know, this is the only paper proving a global well posedness result, for an arbitrarly large power of the nonlinearity, in the smooth Sobolev space $H^s(\mathbb{T}^3)$, $s\ge 2$ (as stated in \cite[Remark 1.1]{Sy} the result can be generalized to the case of $\mathbb{T}^d$, for $d \ge 3$, considering the Sobolev space $H^s(\mathbb{T}^d)$ with $s >\frac d2$). The only restriction lies in the fact that the global existence  result holds for $\mu$-a.e. initial data (that, on the other hand, are not required to be small).
  
  \bigskip 
In the last years the effect of a stochastic perturbation affecting equation \eqref{NLS_det} has been investigated. 
Well-posedness results have been proved for the {\bf stochastic}  NLS equation, that is the NLS equation with a    random forcing, both in additive and multiplicative (It\^o or Stratonovitch) form, by many authors. The first results are due to De Bouard and Debussche
see e.g. \cite{DeBouard+Debussche_1999} and \cite{DeBouard+Debussche_2003}, that, in particular, address the $H^1(\mathbb{R}^d)$-subcritical case, both in the focusing and defocusing case. 
Then, in the same setting, we mention the paper by Barbu, R\"{o}ckner and Zhang \cite{BRZ} and by Hornung \cite{H18}. For the critical case we mention the recent papers \cite{OO, Z}.

In the stochastic setting as well, working in the full space $\mathbb R^d$ is easier than in the periodic setting. 
The papers by Brze{\'z}niak and Millet \cite{BrzezniakStrichartz} and Brze{\'z}niak, Hornung and Weis  \cite{BHW2019} address the case of compact riemannian manifolds dealing with the $H^1$-subcritcal (focusing and defocusing) case.
 To the best of our knowledge, there are no results in the stochastic case for an arbitrary large power of the nonlinearity
 , both in the full space $\mathbb{R}^d$ or in the torus case $\mathbb{T}^d$.

\subsection{No-blow up by noise }

When dealing with the stochastic case, an interesting question is   understanding if the presence of a (suitable) noise guarantees the existence of a global solution, especially in cases where it is known the deterministic counterpart presents blow up phenomena.  
The answer to this question depends strongly on the type of nonlinearity and on the type of noise. 
As mentioned above, when working in the energy space $H^1$, there are blow-up phenomena for the NLS equation \eqref{NLS_det} in the focusing supercritical case. 
In this context, working on the full space $\mathbb{R}^d$,  De Bouard and Debussche  proved in \cite{DeBouard+Debussche_2002, DeBouard+Debussche_2002-PTRF} that, with an additive, nondegenerate and coloured-in-space noise, for any initial data, the blow-up happens before an arbitrary $t>0$ with  positive probability. 
 Still in the focusing supercritical case, the same authors proved in \cite{DeBouard+Debussche_2005} that the conservative, i.e. Stratonovich, noise
accelerates the blow-up with positive probability. The effect of the noise changes completely in the non-conservative case: Barbu et al  studied in \cite{BarbuRZ}  a stochastic NLS equation with linear
non-conservative noise, i.e. not in the Stratonovich form, and proved that, for every  initial condition in $H^1(\mathbb{R}^d)$, the probability of no blow-up converges to $1$, as the noise coefficient tends to infinity.
\\
Another, quite different case is when the noise coefficient $\phi$ is actually a Schr\"odinger operator itself, namely $\phi(u)=i\Delta u$ (modulated Schr\"odinger equation).  
In this case, the noise is also regularizing: Debussche and Tsutsumi showed in \cite{Debussche+Tsutsumi_2011} that this noise makes the 1D quintic 
focusing NLS globally well-posed. 
Chouk and Gubinelli in \cite{Chouk+Gubinelli_2015} extended this result (applying also new techniques \`a la Young integral) to global well-posedness in other regularity classes and to local well-posedness for other critical NLS equations, in dimension $d=1$ and $d=2$; a further extension to the critical case in any dimension is given by Duboscq and R\'eveillac in \cite{DubRev2022}.

The question of the effect of nonlinear noise on blow-up has been widely studied also outside the context of Schr\"odinger equation. The problem is already of interest for finite-dimensional ODEs, where the drift has superlinear growth. In this context, it is known that a suitable noise with superlinear growth can avoid blow-up, even with just a one-dimensional radial structure, see e.g. \cite{AppMaoRod2008}, \cite[Example 5.4]{Gar1988} among many other examples; 
in those cases, no blow-up can be inferred by the application of a Lyapunov function criterion, the Lyapunov function being a logarithmic or power function of the Euclidean norm. There are also situations where blow-up happens for the ODE only along special directions, in such cases blow-up can be avoided even by an additive noise, with a suitable choice of an anisotropic Lyapunov function, see e.g. \cite{AthKolMat2012} again among many other works.
\\
Concerning nonlinear noise in infinite dimension, a case that has been studied is when the noise multiplies the nonlinear flux: regularizing properties of this type of noise have been proved for scalar conservation laws (e.g. Gess and Souganidis \cite{GesSou2017}, Chouk and Gess \cite{ChoGes2019}) and for Hamilton-Jacobi equations (e.g. Gassiat and Gess \cite{GasGes2019}, Gassiat, Gess, Lions and Souganidis \cite{GasGesLioSou2024}), in a similar spirit to case of modulated Schr\"odinger equation. 
The use of a smooth superlinear growth noise - like the one we use here - in the context of SPDEs, seems more recent. Among the first works in this direction, Ren, Tang and Wang considered in \cite{RenTanWan2024} a general framework for no blow-up for SPDEs and used it to show no blow-up by a superlinear noise in two examples, one of which is related to KdV equation. 
Alonso-Oran et al. showed in \cite{AloMiaTan2022} no blow-up for a one-dimensional, transport-type PDE by a superlinear noise. Tang and Wang \cite{tang2022general} and Bagnara, Maurelli and Xu \cite{BagMauXu2023} give general settings for no blow-up by superlinear noise in SPDEs and applied these settings to Euler equations and (for \cite{tang2022general}) other SPDEs in fluid dynamics and KdV equation; see also \cite{CriLan2024}. 
The no blow-up criterion \cite[Theorem 3.1]{tang2022general} could also be applied to the NLS equation, yielding a no blow-up result  similar to the one here though with some technical differences (they would need $s>2+d/2$ and a large noise intensity). Nevertheless, we are not aware of a direct example of no blow-up by superlinear noise in Schr\"odinger equations.

\subsection{Long-time behavior of the solution}

Another important question on stochastic NLS equations is about the effect of the noise on invariant distributions and long-time behaviour. Most 
of the literature considers the case of invariant measures for damped versions of the NLS equation: the idea is that invariant distributions arise in the competition between the noise injecting energy and the damping term dissipating this energy; for example, see \cite{BFZ24, BFZ23, FZ, EKZ, Kim, DO05}. 

Without the damping term, the effect of other noises on invariant distributions and long-time behaviour is different and seems less studied.
 Concerning the modulated Schr\"odinger equation, that is the NLS equation \eqref{NLS_det} with $=i\Delta u$ replaced by 
 with $i\Delta u \partial_t W$, Dumont, Goubet and Mammeri showed in \cite{DumGouMam2021} the exponential decay in $L^\infty$ around the zero solution, for a high-order nonlinearity in the one-dimensional case. 

Outside the context of the  Schr\"odinger equation, in the finite-dimensional setting it is known (see, e.g., \cite{K2}) that the existence of a suitable Lyapunov function implies the existence of an invariant distribution. 
Moreover,  in the infinite-dimensional case there are several examples of noises with stabilizing effects, though in most of them the coefficient in front of the white noise is  constant (i.e., the noise is additive) or a linear or Lipschitz continuous nonlinearity (for a multiplicative noise). We only mention some results, without any claim of completeness. In the celebrated result \cite{ArnCraWih1983} by Arnold, Crauel and Wihstutz, stabilization is achieved by a linear noise which averages over stable and unstable directions.  In a setting closer to ours, Cerrai in \cite{Cer2005} has shown a stabilization effect by a Lipschitz (non conservative) noise for a reaction diffusion equation. We mention also \cite{Tsu2008} by Tsutsumi for a stabilization phenomenon for KdV equation with linear multiplicative noise.
\\
Concerning nonlinear multiplicative noise, in the context of scalar conservation laws, with nonlinear flux multiplying the noise, Gess and Souganidis \cite{GesSou2017} have shown the long-time convergence to the constant solution, which is the unique invariant distribution; the idea in \cite{GesSou2017} is that the noise has an averaging effect over the kinetic random variable. This result has been later extended in \cite{GesSou2017b} to include a deterministic nonlinear flux. Gassiat, Gess, Lions and Souganidis considered in \cite{GasGesLioSou2024} the case of anisotropic flux and get the  convergence in the long-time limit to non-constant stationary solutions.
\\
In the case of SPDEs with a smooth superlinear growth noise, as the one we consider here, to our knowledge the existence and uniqueness of invariant measures has not been proved yet. Crisan and Lang \cite{CriLan2024} mention that, in a future paper, they will show existence of invariant measures for stochastically controlled fluid dynamics models.

\subsection{Main results}

Compared to the literature quoted above, our result present novelities in different directions.    
When compared to the stochastic literature, as far as we know, this is the first result providing the existence of a unique global solution in the regular Sobolev space $H^s(\mathbb{T}^d)$, $s>\frac d2$, for an arbitrary large power of the nonlinear term 
and for any initial data in $H^s(\mathbb T^d)$ (both in the focusing and defocusing case).
When compared to the deterministic literature, our result can be understood as a regularization (no blow-up) by noise result. As mentioned above, for the NLS equation \eqref{NLS_det} it is known that there might be blow-up phenomena, e.g. in the focusing supercritical case.
The superlinear stochastic perturbation $\phi$ we consider in equation \eqref{NLS}  is able to prevent these phenomena without any requirement on  the initial data.  

Our technique to prove the existence of a global solution relies on some ideas contained in \cite{BagMauXu2023} and is based on a tightness argument for the sequence of finite-dimensional Galerkin approximations. 
Roughy speaking, a superlinear noise coefficient "kills the growth" of the nonlinear term  so to get good a priori estimates by 
means of a suitable Lyapunov function. 
This proves the  non-explosion in finite-time.

For what concerns the problem of existence and uniqueness of invariant measures, this is in general a quite challenging problem for the stochastic NLS equation. 
By slightly strengthen the assumptions on the diffusion coefficient $\phi$ and by modifying the Lyapunov function, the
same argument used to prove the existence of global solutions, is  adapted
first to prove the existence of invariant measures supported on $H^s(\mathbb T^d)$. Then, when $\phi(u)=f(u)u$
for suitable $f$, we prove that the zero solution is a global attractor so that 
 $\mu=\delta_0$ is the unique invariant measure. Roughly speaking, the (even stronger)  superlinear noise coefficient "force the dynamics" to converge to the zero solution, which is an equilibrium of the system. As far as we know, this is the first result proving the existence and uniqueness of the invariant measure (as well as some stability results) for the NLS on $\mathbb{T}^d$ for an arbitrary large power of the nonlinearity.

\smallskip

The paper is organized as follows. In Section \ref{math_ass_sec} we introduce the mathematical framework, state the assumptions and formulate the main results. Section \ref{S-wellposedness} is devoted to the proof of the existence of a unique global solution.  First we introduce the sequence of Galerkin approximations, then we prove its tightness and  convergence to a martingale solution of the NLS equation \eqref{NLS}. Moreover pathwise uniqueness is obtained.
In Section \ref{erg_res_sec} we prove the existence and uniqueness of the invariant measure and some stability result for the zero solution. In Appendix \ref{tight_sec_main} we collect some compactness and tightness criteria, whereas in Appendix \ref{tec_lem_sec} we provide the proof of a technical lemma.
In Appendix \ref{manifolds} we provide a generalization of our results in the case of Riemannian manifolds.

\section{Mathematical setting, assumptions and main results}
\label{math_ass_sec}

Let $\mathbb{T}^d:=(\mathbb{R}/2\pi\mathbb{Z})^d$. For $p \in[ 1, \infty)$ we denote by $L^p:=L^p(\mathbb{T}^d)$ the Lebesgue space of all complex-valued measurable $p$-integrable functions on $\mathbb{T}^d$. By $L^\infty:=L^\infty(\mathbb{T}^d)$ we denote the Banach space of Lebesgue measurable essentially bounded complex-valued functions.

We expand a periodic function in Fourier series as 
\begin{equation*}
u(x)=\frac{1}{(2\pi)^{\frac d2}}\sum_{k \in \mathbb{Z}^d}\widehat{u}(k)e^{\im k \cdot x}, \qquad \widehat{u}(k):= \frac{1}{(2 \pi)^{\frac d2}}\int_{\mathbb{T}^d}u(x)e^{-\im k \cdot x }\, {\rm d}x.
\end{equation*}
We set $\langle k\rangle:=\sqrt{1+|k|^2}$ for $k \in \mathbb{Z}^d$.
For $s\in \mathbb{R}$ we define $J^s:=(I-\Delta)^{\frac s2}$ as the operator defined in terms of Fourier series
\begin{equation*}
J^su(x)=\frac{1}{(2 \pi)^{\frac d2}}\sum_{k \in \mathbb{Z}^d} \langle k\rangle^s\widehat{u}(k)e^{\im k \cdot x}.
\end{equation*}
For $p \in [1, \infty)$ we define the \textit{
Bessel potential space} $H^{s,p}:=H^{s,p}(\mathbb{T}^d)$ as the space of all distributions $u$ such that $J^s u \in L^p$. 
We have that $J^\sigma$ is an isomorphism between the spaces $H^{s,p}$ and $H^{s-\sigma, p}$.
The norm in $H^{s,p}$ will be denoted by 
\begin{equation*}\|u\|_{H^{s,p}}:=\|J^s u\|_{L^p}.
\end{equation*}
For $p=2$ the space $H^s:=H^{s,2}$ is a complex  Hilbert space with inner product 
\begin{equation}\label{H_s_norm}
(u,v)_s = \sum_{k \in \mathbb{Z}^d} \langle k\rangle^{2s} \widehat{u}(k) \overline{\widehat{v}(k)}
\end{equation}
where $\overline v$ denotes the complex conjugate of $v$.
We  abbreviate ${H}:=L^2$.

We have the equivalence of the norms (see e.g. \cite[Remark 2.3]{CG19}) 
\begin{equation}
\label{eq_norms}
\|u\|_{s} \simeq \|u\|_{L^2}+ \|(-\Delta)^{\frac s2}u\|_{L^2}.
\end{equation}
\\
The dual of the space $H^s$ is $H^{-s}$. We denote by $_{H^{-s}}\langle \cdot, \cdot \rangle_{H^s}$ the $H^s-H^{-s}$-duality bracket; it reduces to the $H$-scalar product  $(u,v)_0$ when $u \in H$ and $v\in H^s$.

The space $H^1$ is usually called the energy space (see, e.g., \cite{Cazenave}). Differently from the approach 
to prove global existence by using the a priori estimates of the mass and of the energy, we will be able to work in any space 
$H^s$ with  $s>\frac d2$. The main point in considering $s>\frac d2$ is that $H^s\subset L^\infty$.
Actually, two parameters will play a crucial role later on: $s$ and $s'$ such that
\begin{equation}
\label{STAR}
s>s'>\frac d2,
\end{equation}
so there are  the  continuous embeddings 
\begin{equation}
\label{big_emb}
H^{s+1}\subset H^s \subset H^{s'} \subset H \simeq H^* \subset H^{-s'}\subset H^{-s}\subset H^{-s-1},
\end{equation}
\begin{equation}
\label{Sobolev}
H^s \subset H^{s'} \subset L^\infty,
\end{equation}
and the compact embedding
\begin{equation}
H^s \subset H^{s'} .
\end{equation}
It follows that  the embedding $H^{s} \subset H^{-s-1}$ is compact as well.

We will deal with the following functional spaces for a given $r>0$: 
\begin{itemize}
\item
$C_w([0,T];H^r)$ with the topology generated by the family of seminorms 
$\|u\|_{k,v}:=\displaystyle\sup_{t \in [0,k]}|\langle u(t),v\rangle|$, $k \in \mathbb{N}$, $v \in H^{-r}$; 
\item
$L^\infty(0,T;H^r)$ with the topology induced by the norm
 $\displaystyle\operatorname*{esssup}_{t \in [0,T]}  \|u(t)\|_{r}$; 
\item
$C([0,T];H^r)$ with the topology induced by the norm $\displaystyle\sup_{t \in [0,T]}\|u(t)\|_{r}$; 
\item
$C^{0,\beta}([0,T];H^{-r})$ with the topology induced by the norm 
\begin{equation}\begin{split}
\label{norm_C}
\|u\|_{C^{0,\beta}([0,T];H^{-r})}
&:=\|u\|_{L^\infty(0,T;H^{-r})}+ [u]_{C^{0, \beta}([0,T];H^{-r})}
\\
&:=\displaystyle\operatorname*{esssup}_{t \in [0,T]}\|u(t)\|_{-r}+ \sup_{s, t \in [0,T]; s \ne t}\frac{\|u(t)-u(s)\|_{-r}}{|t-s|^{\beta}}.
\end{split}
\end{equation}
\end{itemize}
From \cite{Strauss} we know that for any $s<r$ we have
\[
C_w([0,T];H^s)\cap L^\infty(0,T;H^r)= C_w([0,T];H^r).
\]

\begin{notation}
 We write $u=\Re  u +\im \Im u$ to specify the real and imaginary parts of $u\in \mathbb C$.
 \\
 We write $a\lesssim b$ when there exists a constant $C$ such that $a\le Cb$; to highlight the role of this constant, we write $a\lesssim_C b$. 
 \\
Given two Hilbert spaces $E$ and $F$, we denote by $\mathcal{L}(E,F)$ the space of all linear bounded
operators $B: E\to F$ and abbreviate $\mathcal{L}(E):=\mathcal{L}(E,E).$
 \\ 
 We denote the open ball centered at the origin with radius $R$ by $\mathbb{B}_R$ and its  complement by $\mathbb{B}^c_R$. In order to emphasize the underlying topology, we write $\mathbb{B}_{R, \mathcal{H}}$, i.e. $\mathbb{B}_{R, \mathcal{H}}$ is a subset of $\mathcal{H}$ and the radius $R$ of the ball is measured in the $\mathcal{H}$-norm. 
\end{notation}


\subsection{The linear operator}
We set
\[
A=-\Delta,\qquad \mathcal D(A)=H^2.
\]
$A$ is a non-negative self-adjoint operator on $H$, densely defined. 
We will consider it as  a linear self-adjoint  operator in $H^s$ with domain 
$H^{s+2}$.
Moreover, it defines a unitary $C_0$-group
$(e^{itA})_{t\in \mathbb R}$ in $H^s$.

In the space $H$ 
 we consider the complete orthonormal basis $\{e_k\}_{k\in \mathbb Z^d}$
of the eigenfunctions of the operator $A$. We have
$Ae_k= |k|^2e_k$, with
\[
e_k(x)= \frac{1}{(2\pi)^{\frac d2}}e^{\im k\cdot x}
\]
So the eigenvalues  are  $\lambda_k=|k|^2$. Moreover,  $\|e_k\|_s=\langle k\rangle^{s}$.

 Later on we will need the finite dimensional operator 
 $P_n:H\to H_n=Span\{e_k: |k|\le n\}$  defined by
\begin{equation}
\label{P_n_def}
P_n  u = \sum_{k:|k|\le n} (u,e_k)_0 \ e_k.
\end{equation}

\subsection{The nonlinear operator}
We write the nonlinearity as 
\begin{equation}
\label{F}
F(u):= |u|^{2\sigma}u
\end{equation}
and assume that
\begin{equation}\label{sigma-intero}
\sigma \in \mathbb N.
\end{equation}
In Remark \ref{non-intero}  we consider non integer values of $\sigma$.

In the following Lemma we collect some properties of the nonlinear operator $F$. The proof of the result relies on the Moser Estimate (see \cite{Benyi+Oh+Zhao})
that we recall here in the form more suitable for our needs. Let $r \ge 0$, then 
\begin{equation}
\label{Moser}
\|v_1 v_2\|_{r} \lesssim_{r,d} \|v_1\|_{L^\infty}\|v_2\|_{r}+\|v_1\|_{r} \|v_2\|_{L^\infty}, 
\qquad \forall \ v_1, v_2 \in L^\infty \cap H^r.
\end{equation}

\begin{lemma}
\label{Lemma_F}
\begin{itemize}
\item [(i)] Let $r \ge 0$. Then $F$ maps the space $H^r \cap L^\infty$ into $H^r\cap L^\infty$ and 
\begin{equation}
\label{F_infty}
\|F(u)\|_{r} \lesssim_{r,d,\sigma} \|u\|^{2\sigma}_{L^\infty}\|u\|_{r}, \qquad u \in H^r \cap L^\infty.
\end{equation}
\item [(ii)]
Let $r>\frac d2$. Then $F$ maps $H^r$ into $H^r$ and, for any $u, v \in H^r$, it holds 
\begin{equation}
\label{F3}
\|F(u)-F(v)\|_{r}
\lesssim_{r,d,\sigma}  \left( \|u\|_r^{2\sigma}+  \|v\|_{r}^{2\sigma}\right)
                \|u-v\|_{r} .
\end{equation} 
\end{itemize}
\end{lemma}
\begin{proof}
\begin{itemize}
\item [(i)] We proceed by an induction argument on $\sigma$.
For $\sigma=1$, using twice the Moser inequality we estimate
\begin{align*}
\||u|^2 u\|_{r} 
&\lesssim_{r,d}\|u\|^2_{L^\infty}\|u\|_{r} + \||u|^2\|_{r}\|u\|_{L^\infty}
\\
&\lesssim_{r,d} \|u\|^2_{L^\infty}\|u\|_{r} + \left( \|u\|_{L^\infty}\|\bar u\|_{r}+ \|\bar u\|_{L^\infty}\|u\|_{r}\right)\|u\|_{L^\infty}
\\&
\lesssim_{r,d} \|u\|^2_{L^\infty}  \|u\|_{r}.
\end{align*}
Suppose  for some $\sigma\ge 1$ it holds 
\begin{equation}
\label{ind}
\||u|^{2\sigma}u\|_{r} \lesssim_{r,d}\|u\|^{2\sigma}_{L^\infty}\|u\|_{r};
\end{equation}
let us prove the same estimate for $\sigma +1$. 
Using repeatedly the Moser inequality and the inductive hypothesis \eqref{ind}, we get
\begin{align*}
\||u|^{2(\sigma+1)}u\|_{r} 
&\lesssim_{r,d}\| |u|^{2\sigma} u\|_{L^\infty} \||u|^2\|_{r} + \||u|^{2\sigma }u\|_{r}\||u|^2\|_{L^\infty}
\\
&\le \|u\|_{L^\infty}^{2\sigma +1}2\|u\|_{L^\infty} \|u\|_{r}+ \|u\|^{2\sigma}_{L^\infty} \|u\|_{r} \|u\|^2_{L^\infty}
\\
& \lesssim_{r,d}\|u\|^{2(\sigma +1)}_{L^\infty}\|u\|_{r}.
\end{align*}
\item [(ii)] 
We exploit the estimate
 \begin{equation}
\label{stima_F_L_infty}
\left| |a|^{2\sigma}a-|b|^{2\sigma}b\right |\lesssim_\sigma (|a|^{2\sigma}+|b|^{2\sigma})|a-b|,
\end{equation}
valid for any $a,b\in \mathbb C$. Since $H^r$ is a multiplicative algebra for $r>\frac d2$ we easily get
\begin{align*}
\|F(u)-F(v)\|_{r} 
&\lesssim_{r, d,\sigma} \left( \| |u|^{2\sigma} \|_{r} + \| |v|^{2\sigma} \|_{r} \right) \|u-v\|_{r}
\\
&\lesssim_{r, d,\sigma}\left( \| u \|_{r}^{2\sigma} + \| v \|_{r}^{2\sigma} \right) \|u-v\|_{r}.
\end{align*}
\end{itemize}
\end{proof}

From \eqref{F_infty} we infer the existence of a positive constant $K=K(s, \sigma,d)$ such that 
\begin{equation}\label{stimaF}
\|F(u)\|_{s}\le K  \|u\|^{2\sigma}_{L^\infty}\|u\|_{s}, \qquad \forall \ u \in H^s\cap L^\infty, s>0.
\end{equation}

\begin{remark}\label{non-intero}
We can consider $\sigma>0$ not integer but in this case the results are less general than in Lemma \ref{Lemma_F}.
Indeed, the mapping $F:\mathbb C\to\mathbb C$,  defined as $F(u)=|u|^{2\sigma}u$, has continuous derivatives of all orders when 
$\sigma$ is an integer number, and till the order $1+\lfloor 2\sigma\rfloor$ when $\sigma$ is not integer.
Following   Lemma 4.10.2 of \cite{Cazenave}  when $\sigma\notin \mathbb N$
 we can deal with the $H^r$-regularity  of $F(u)$ for  $\frac d2<r<1+\lfloor 2\sigma\rfloor$ and $r$ integer. 
 \end{remark}

\subsection{The stochastic forcing term}
\label{sec:noise}

To prove the existence of a unique strong (in the probabilistic sense) solution to \eqref{NLS} we work under the following assumptions on the noise. 

We say  that  a filtered probability space $\bigl(\Omega, \F, \Filtration, \Prob \bigr)$   satisfies the standard conditions when 
 the filtration $\Filtration=\bigl(\mathcal{F}_t\bigr)_{t\ge 0}$ is  right-continuous 
 and  all $\Prob$-negligible  sets of $\F$ are elements of $\mathcal{F}_0$.
\\
Fix $s' $ and $s$ as in \eqref{STAR}.

\begin{enumerate}[label = \textbf{(H\arabic*)}, ref = \textbf{(H\arabic*)}]
\item\label{H1} 
 $W$ is a real-valued one-dimensional Brownian motion on a filtered probability space $(\Omega, \mathcal{F}, \mathbb{F}:=\{\mathcal{F}_t\}_{t \ge 0}, \mathbb{P})$ satisfying the standard assumptions.
\item\label{H2}
The diffusion coefficient $\phi$ is such that  
\begin{itemize}
\item [i)]
$\phi: H^{s} \rightarrow H^{s}$ is bounded on balls, 
\item [ii)]$\phi: H^{s'} \rightarrow H^{-s-1}$ is  continuous and bounded on balls; 
\end{itemize}
\item \label{H3} the projected coefficient $P_n\phi:H_n \rightarrow H_n$ is locally Lipschitz continuous for any $n \in \mathbb{N}$;
%
\item \label{H4} there exists a measurable function $\psi:\mathbb{R}^+\times \mathbb{R}^+\rightarrow \mathbb{R}^+$ which is locally bounded (i.e. bounded on bounded sets) and such that
\begin{equation*}
\|\phi(u)-\phi(v)\|_{s} \le \psi(\|u\|_{s}, \|v\|_{s})\|u-v\|_{s}
\qquad \forall \ u,v\in H^s;
\end{equation*}
\item\label{H5}
there exist $r>1$  and $B\in\mathbb R$ such that
\begin{equation}  
           \label{stima-K-phi}
 | \alpha | K \|u\|_{L^\infty}^{2\sigma} 
 +
 \frac 12 \frac {\| \phi(u)\|_{s}^2}  {\|u\|^2_{s} }
 - \frac{ [\Re \big(u, \phi (u)\big)_{s}]^2 } {\|u\|_{s}^4}\le B
 \qquad \forall \ u \in \mathbb B^c_{r,H^s}
\end{equation}
where $K$ is the constant appearing in \eqref{stimaF}.
\end{enumerate}

\begin{remark}
\label{loc_Lip}
Working under assumptions \ref{H1}, \ref{H2}, \ref{H3}
and \ref{H5} is enough to prove the existence of  global-in-time martingale solutions.
Assumption \ref{H4} is only needed to prove pathwise uniqueness. 

Notice that Assumption \ref{H4} implies Assumption \ref{H3} due to the equivalence of norms in finite dimensional spaces. 
We explicitly mention \ref{H3} to make clear in what follows what assumptions are needed to infer the existence of  global-in-time 
martingale solutions and what are needed to infer the pathwise uniqueness.
\end{remark}

\begin{remark}
    Assumptions \ref{H5} could be weaken: see Remark \ref{weaker_H3}. The auxiliary space $H^{s'}$, that appears in Assumption \ref{H2}(ii), plays a role in the construction of martingale solutions, see Remark \ref{Hs'_rem}.
\end{remark}

To prove the existence of invariant measures $\mu$ for \eqref{NLS} we need to strengthen Assumption \ref{H5} as follows.

\begin{enumerate}[label = \textbf{(H5')}, ref = \textbf{(H5')}]
\item \label{H5bis} 
There exist  $p\in (0,1)$, $r>1$  and $B<0$ such that
\begin{equation}\label{seconda-ipotesi-phi}
 | \alpha | K   \|u\|_{L^\infty}^{2\sigma}+\frac 12 \frac{\| \phi(u)\|_{s}^2}{ \|u\|^2_s}
 -\frac{2-p}2\frac{ [\Re \big(u, \phi (u)\big)_{s}]^2 } {\|u\|_{s}^4}
 \le B \qquad \forall \ u \in \mathbb B^c_{r,H^s}.
\end{equation}
\end{enumerate}

To prove that $\mu=\delta_0$ is the unique invariant measure and the zero solution is asymptotically stable, we will strengthen Assumption \ref{H5} as follows.

\begin{enumerate}[label = \textbf{(H5'')}, ref = \textbf{(H5'')}]
\item \label{H5bisbis} 
There exists $f:H^s \rightarrow \mathbb{C}$ such that $\phi(u)=f(u)u$.
Moreover, there exist  $p\in (0,1)$ and $B<0$ such that
\begin{equation}\label{terza-ipotesi-phi}
|\alpha|  K    \|u\|_{L^\infty}^{2\sigma} + \frac 12 |f(u)|^2 -\frac{2-p}{2} [\Re f(u)]^2  \le B, \quad \forall\  u \in H^s.
\end{equation}
\end{enumerate}
Notice that \ref{H5bisbis} is stronger that  \ref{H5bis}, which is stronger than \ref{H5}. 
Anyway, we presented all these three assumptions, because we gradually
prove stronger  and stronger results: no blow-up for solutions under \ref{H5}, existence of invariant measures under \ref{H5bis}, and asymptotic stability of the zero solution as well as the uniqueness of the invariant measure $\mu=\delta_0$ under \ref{H5bisbis}. 

Finally, rewriting \ref{H5bisbis} as
\[
|\alpha|  K    \|u\|_{L^\infty}^{2\sigma} + \frac 12 [\Im f(u)]^2 \le\frac{1-p}{2} [\Re f(u)]^2  + B, \quad \forall\  u \in H^s
\]
we realize that  a key role is played by the real part of $f$, which dominates the intensity of its imaginary part and of the nonlinear term.

\subsection{Examples of noise forcing term}
\label{example_noise_sec}

Let us provide  examples of function  $\phi$ appearing in the noise term.

We introduce the function $h: \mathbb{R}^+\rightarrow \mathbb{C}$ as 
\begin{equation*}
    h(x)=a (1+ x)^{b}+\im  c (1+ x)^d, 
\end{equation*}
with $a \ne 0$, $c \in \mathbb{R}$ and $b, d>1$.
We have that 
\begin{equation}
    \label{ex_0}
x \mapsto |h(x)|\quad \text{is increasing}. 
\end{equation}
Moreover, $h$ is continuously differentiable with 
\begin{equation}
    \label{ex_1}
x \mapsto |h'(x)|\quad \text{ increasing}.  
\end{equation}
As a consequence of the Mean Value Theorem we obtain the estimate 
\begin{equation}
\label{ex_2}
|h(x)-h(y)| \le \left(|h'(x)|\vee |h'(y)|\right)|x-y|, \qquad x, y \in \mathbb{R}^+.    
\end{equation}
Let now $f:H^s \rightarrow \mathbb{C}$ be defined as 
\begin{equation} 
\label{f_es_1}
f(u)=h(\|u\|_{L^\infty});
\end{equation}
we consider a diffusion term $\phi$ of the following form
\begin{equation}
\label{ex_noise_phi}
\phi(u)= f(u) u.
\end{equation}
Let us verify that this  choice  satisfies Assumptions \ref{H2},  \ref{H4} and \ref{H5bisbis}.
These in turns imply \ref{H3}, \ref{H5} and \ref{H5bis}.

\begin{itemize}
\item [\ref{H2}] Using the Sobolev embeddings \eqref{Sobolev} it is easy to see that $\phi$ is bounded on balls as a map from $H^s$ into itself and from $H^{s'}$ into $H^{-s-1}$.
\\
Let us prove that  $\phi: H^{s'} \rightarrow H^{s' }$ is continuous; this implies that 
$\phi: H^{s'} \rightarrow H^{-s-1}$ is continuous as well. Let $u, v \in H^{s'}$.
Bearing in mind \eqref{ex_2} and \eqref{Sobolev} we estimate
\begin{align*}
\|\phi(u)-\phi(v)\|_{s'} 
&\le |f(u)-f(v)|\|u\|_{s'}+|f(v)|\|u-v\|_{s'} 
\\
&\le
\left[ |h'(\|u\|_{L^\infty})| \vee  |h'(\|v\|_{L^\infty})|\right] \left| \|u\|_{L^\infty}- \|v\|_{L^\infty}\right| \|u\|_{s'}+
|h(\|v\|_{L^\infty})| \|u-v\|_{s'}
\\
&\lesssim \|u-v\|_{s'} \left( \|u\|_{s'} [|h'(\|u\|_{L^\infty})| \vee  |h'(\|v\|_{L^\infty})| ]+|f(v)|\right)
\\
&\lesssim   \psi (\|u\|_{s'}, \|v\|_{s'})\|u-v\|_{s'},
\end{align*}
where $\psi:\mathbb{R}^+\times \mathbb{R}^+\rightarrow \mathbb{R}^+$ is a locally bounded map. In the last equality we used \eqref{ex_0}, \eqref{ex_1} and \eqref{Sobolev}. 
\item[\ref{H5bisbis}] Setting $\tilde B=-B>0$
estimate \eqref{terza-ipotesi-phi} becomes 
\begin{equation} 
|\alpha| K \|u\|_{L^\infty}^{2\sigma} 
 +
 \frac 12[\Im f(u)]^2
 \le 
 \frac{(1-p)}2[\Re f(u)]^2- \tilde B.
\end{equation}
Therefore, \eqref{terza-ipotesi-phi} is fulfilled when there exists a positive constant $\tilde B$ such that
\begin{equation} 
\label{K_2}
2|\alpha|K  \|u\|_{L^\infty}^{2\sigma} + c^2(1+  \|u\|_{L^\infty})^{2d} 
\le
(1-p) a^2  (1+\|u\|_{L^\infty})^{2b}- 2\tilde B
\end{equation}
for any $u \in H^s$.
We then choose the parameters $a,b,c,d$ as follows. 
\begin{itemize}
\item  
When 
$b=d=\sigma\ge 1$, for a fixed $p\in (0,1)$ we require $a$ and $c$ to fulfil 
\[
2|\alpha|K  + c^2 < (1-p)a^2.
\]
In this way 
$2\tilde B:=(1-p)a^2-2|\alpha|K  - c^2>0$.
\item
When $b\ge d>\sigma\ge 1$, we use
\[
\|u\|_{L^\infty}^{2\sigma}\le (1+\|u\|_{L^\infty})^{2d}
\le (1+\|u\|_{L^\infty})^{2b}
\]
and we can proceed in a similar way to the previous case.
\end{itemize}
\item [\ref{H4}] One argues as done for \ref{H2}.
\end{itemize}

Notice that \ref{H5} would require less restrictive assumptions on the parameters than \ref{H5bisbis}. 
Indeed, estimate \eqref{stima-K-phi} becomes 
\begin{equation} 
|\alpha| K \|u\|_{L^\infty}^{2\sigma} 
 +
 \frac 12[\Im f(u)]^2
 \le 
 \frac12[\Re f(u)]^2+ B.
\end{equation}
Therefore, \eqref{stima-K-phi} is fulfilled when there exists $B\in \mathbb R$ such that
\begin{equation} 
\label{K_new}
|\alpha| K  \|u\|_{L^\infty}^{2\sigma} + \frac {c^2}{2}(1+  \|u\|_{L^\infty})^{2d} 
\le
\frac{a^2}{2} (1+\|u\|_{L^\infty})^{2b}+B.
\end{equation}
We then choose the parameters as follows. 
\begin{itemize}
\item
When 
$b=d=\sigma\ge 1$,  we require   $2|\alpha|K  + c^2 \le  a^2$; then $B$ can be any non-negative constant.
 \item  
When $b=d>\sigma\ge 1$, we use Young inequality to get that for any $\epsilon>0$ there exists $C_\epsilon>0$ such that
\[
2|\alpha|K\|u\|_{L^\infty}^{2\sigma}\le \epsilon  \|u\|_{L^\infty}^{2b}+C_\epsilon\le
\epsilon (1+ \|u\|_{L^\infty})^{2b}+C_\epsilon.
\]
Therefore \eqref{K_new} holds when $c^2< a^2$ and $B$ is chosen large enough, i.e. $B\ge C_\epsilon$ when we choose $\epsilon\le a^2-c^2$ (notice that 
$C_\epsilon$ depends also on $\alpha, K, b,\sigma$).
\item 
When $b>d>\sigma\ge 1$, we use the Young inequality to bound also the power $2d$, i.e.
\[
2|\alpha| K\|u\|_{L^\infty}^{2\sigma} + c^2 (1+  \|u\|_{L^\infty})^{2d} \le
\epsilon (1+ \|u\|_{L^\infty})^{2b}+C_\epsilon.
\]
Therefore \eqref{K_new} holds for every $a\ne 0$ and $B$ chosen large enough 
as in the previous case.
%
\end{itemize}


\begin{remark}
    Another possible example is given by a diffusion term $\phi$ of the form \eqref{ex_noise_phi}, with $f$ as in \eqref{f_es_1} and 
    \begin{equation*}
h(x)=ax^b+\im cx^d,
    \end{equation*}
    for $a\ne 0$, $c \in \mathbb{R}$ and $b,d>1$.
\end{remark}

\begin{remark}
One could consider a slightly more general example for the noise: $\phi$ is of the form \eqref{ex_noise_phi} with $f:H^s \rightarrow \mathbb{C}$ given by
\begin{equation} 
f(u)=a (1+ \|u\|_{X_1})^{b}+\im  c(1+ \|u\|_{X_2})^{d}, 
\end{equation}
with $a \ne 0$, $c\in \mathbb{R}$ and $b, d>1$. 
\\
The spaces $X_1$ and $X_2$ are chosen in such a way that for some constants $K_0$, $K_1$ and $K_2$
\[\begin{split}
&\|u\|_{X_2}\le K_2 \|u\|_{X_1}\le  K_1  \|u\|_{s'} \qquad \forall u \in H^s 
\\
&\|u\|_{L^\infty}\le K_0 \|u\|_{X_1}\qquad \forall u \in H^s ,
\end{split}
\]
where, as usual, we consider $\frac d2<s'<s$. 
\\
To show that conditions \ref{H5}--$\ref{H5bisbis}$ are satisfied one imposes conditions on the parameters  $a,b,c,d$ that involve also the constants $K_0, K_1$ and $K_2$.
\end{remark}

\subsection{Statement of the main results}

The following hypotheses stand throughout the paper
\begin{equation}\label{sigma+n}
\sigma \in \mathbb N, \qquad  s>\frac d2.
\end{equation}

We summarize our main results as follows. 

\begin{theorem}
\label{mainTH}
Assume  \eqref{sigma+n} and \ref{H1}-\ref{H5}. Then, for any initial datum $u^0 \in H^s$ there exists a unique global-in-time strong solution to \eqref{NLS_abs} with $\mathbb{P}$-a.s. paths in $C([0,\infty);H^s)$.
\end{theorem}

 \begin{theorem}
\label{misura-invariante}
Assume  \eqref{sigma+n} and  \ref{H1}-\ref{H5bis}. 
Then there exists at least one invariant measure $\mu$ for equation \eqref{NLS_abs}, supported in $H^s$.
\end{theorem}

 \begin{theorem}
\label{unic-misura-invariante}
Assume  \eqref{sigma+n} and \ref{H1}-\ref{H5bisbis}. Then $\mu=\delta_0$ is the unique invariant measure for equation \eqref{NLS_abs} and the zero solution is exponentially stable.
\end{theorem}

For a more precise statement of the results concerning the ergodic properties to the solution process see Theorems \ref{misura-invariante} and \ref{uniq_inv_thm}.

\section{Existence and uniqueness of global strong solutions}
\label{S-wellposedness}

In this section we prove Theorem \ref{mainTH}, that is that equation \eqref{NLS} admits a unique strong (in the probabilistic sense) solution. 

We rewrite equation \eqref{NLS} in the following abstract form 
\begin{equation}
\label{NLS_abs}
\begin{cases}
{\rm d} u(t)+\im  \left[ -A u(t)+\alpha  F(u(t))  \right] \,{\rm d}t
= \phi(u(t))  \,{\rm d} W(t) , \qquad t>0
\\
u(0)=u^0.
\end{cases}
\end{equation}


We consider martingale solutions, i.e. weak solutions in the probabilistic sense, and strong solutions as well.
Here are the definitions.


\begin{definition}[martingale solution]\label{def-martingale solution}
Let $u^0 \in H^s$.
A \emph{martingale solution} of the equation $\eqref{NLS_abs}$ on the time interval $[0,\infty)$ 
with initial datum $u^0$ is a system
$
\bigl(\tilde{\Omega},\tilde{\F},\tilde{\mathbb P},\tilde{\Filtration},\widetilde{W},u\bigr)
$
 consisting of
	\begin{itemize}
		\item a filtered probability space $\bigl(\tilde{\Omega},\tilde{\F},\tilde{\Prob},\tilde{\Filtration}\bigr)$  satisfying the standard conditions; 
		\item  a real valued one dimensional Brownian motion $\widetilde{W}$  on  $\bigl(\tilde{\Omega},\tilde{\F},\tilde{\Prob},\tilde{\Filtration}\bigr);$
		\item an $H$-valued 
		 continuous and  $\tilde{\Filtration}$-adapted process  $u$ with
$\tilde{\Prob}$-almost all paths in $ C_w([0,\infty);H^s)$,
	such that for every $t\in [0,\infty)$
 the equality
	\begin{equation}\label{eqn-ItoFormSolution}
	u(t)=  u^0+ \im\int_0^t \left[ A u(s)-\alpha F(u(s))\right] \df s     
	+\int_0^t \phi(u(s))\, {\rm d} \widetilde{W}(s)
	\end{equation}
	holds $\tilde{\mathbb{P}}$-almost surely in $H^{-s-1}$.  
\end{itemize}
\end{definition}

\begin{definition}
Let $u^0 \in H^s$.
Assume that $
\bigl({\Omega},{\F},{\mathbb P},{W},{\Filtration}\bigr)
$
is a system consisting of 
\begin{itemize}
		\item a filtered probability space $\bigl({\Omega},{\F},{\Prob},{\Filtration}\bigr)$  satisfying the standard conditions;
		\item  a real valued one dimensional Brownian motion ${W}$  on  $\bigl({\Omega},{\F},{\Prob},{\Filtration}\bigr)$.
\end{itemize}
A \emph{strong solution} of the equation \eqref{NLS_abs}  on the time interval $[0,\infty)$ with the initial datum $u^0$ is a $H$-valued continuous and  $\Filtration$-adapted process  $u$ with
${\Prob}$-almost all paths in $ C_w([0,\infty);H^s)$,
	such that for every $t\in [0,\infty)$
 the equality
	\begin{equation}\label{eqn-ItoFormSolution}
	u(t)=  u^0+ \im\int_0^t \left[ A u(s)-\alpha F(u(s))\right] \df s     
	+\int_0^t \phi(u(s))\, {\rm d} W(s)
	\end{equation}
	holds ${\mathbb{P}}$-almost surely in in $H^{-s-1}$ 
\end{definition}


\begin{remark}
\label{rem_in_data}
    We will deal with deterministic intial data, the extension to random initial data is straightforward. 
\end{remark}

To prove Theorem \ref{mainTH} we proceed as follows. In subsection \ref{sect:Galerkin} we introduce the finite-dimensional Galerkin approximation of equation \eqref{NLS_abs}. In Propositions \ref{p_tight_1} and \ref{p_tight_2} we prove uniform a-priori bounds for the solution of the approximated problem. These bounds allow to infer the tightness of the laws defined by the solution of the Galerkin approximation (see Proposition \ref{prop_tight}).
In section \ref{suc_sec_con} we prove the convergence of
the Galerkin approximations to the martingale solution of problem \eqref{NLS_abs}. In subsection \ref{path_uniq_sec} we prove pathwise uniqueness of the solution, from which we also infer that the solution is strong in the probabilistic sense.

\subsection{The Galerkin approximation}
\label{sect:Galerkin}

We introduce a finite-dimensional approximation of equation \eqref{NLS_abs}.
To this end we introduce  
 the projector operator $P_n:H\to H_n=Span\{e_k: |k|\le n\}$ introduced in 
\eqref{P_n_def}.
For any $r\ge 0$ we have 
\begin{equation}
\label{bound_H^s_Pn}
(P_nu, v)_{r}=( u, P_nv)_{r}, \qquad  \|P_n u\|_{r}
\le \|u\|_{r}, \quad  \forall  u, v \in H^r
\end{equation}
and 
\begin{equation}\label{limit_H^s}
\lim_{n\to\infty} \|P_n-I\|_{\mathcal{L}(H^{r})}=0.
\end{equation}
By density, we can extend $P_n$ to an operator $P_n :H^{-r}\to H_n$ with 
\begin{equation}
\label{P_nU'}
\|P_n\|_{\mathcal{L}(H^{-r})}\le 1.
\end{equation} 

We consider the Faedo-Galerkin approximation in the space $H_n$, obtained by projecting the NLS equation \eqref{NLS_abs} onto the finite dimensional space $H_n$:
\begin{equation}
\label{Galerkin}
\begin{cases}
d u_n(t)+ \im \left[- A  u_n(t)+\alpha  P_nF(u_n(t)) \right] \,{\rm d}t 
= P_n \phi(u_n(t))  \,{\rm d} W(t)
\\
u_n(0)=P_n(u^0),
\end{cases}
\end{equation}

It is a classical result to show that the Faedo-Galerkin equation has a unique solution. This is a strong solution in the probabilistic sense. First we consider  local existence.

\begin{proposition}
\label{local_sol}
Assume  \eqref{sigma+n}, \ref{H1} and \ref{H3}. Then for any $n \in \mathbb{N}$ and 
 $u^0 \in H^{s}$   there exists a unique local solution $u_n$ of \eqref{Galerkin} with continuous paths in $H_n$ and maximal existence time $\tau_n$, which is a blow-up time in the sense that 
\begin{equation*}
\limsup_{t \rightarrow \tau_n(\omega)}\|u_n(t,\omega)\|_{H_n}=+ \infty
\end{equation*}
for $\mathbb{P}$-a.e. $\omega \in \Omega$ with $\tau_n(\omega)<\infty$.
\end{proposition}
\begin{proof}
The result is a consequence of the well-known theory for finite-dimensional stochastic differential equations with locally Lipschitz continuous coefficients.

 Let $n \in \mathbb{N}$, we are done if we prove that the functions 
 \begin{equation*}
 h_n(x):= -\im A x + \im P_n (F(x)), \qquad l_n(x):= P_n \phi(x), \quad x \in H_n,\end{equation*}
 are locally Lipschitz continuous in $H_n$. Let $x, y \in H_n$ with $\|x\|_{0} \le R$ and $\|y\|_{0} \le R$. Using \eqref{bound_H^s_Pn}, \eqref{F3}, \eqref{Sobolev} and the equivalence of norms in $H_n$, we estimate 
 \begin{align*}
 \|P_nF(x)-P_nF(y)\|_{0}& \le \|F(x)-F(y)\|_0
 \lesssim \left( \|x\|_{L^\infty}^{2\sigma}+\|y\|_{L^\infty}^{2\sigma}\right) \|x-y\|_0
 \\
 &\lesssim \left( \|x\|_{s}^{2\sigma}+\|y\|_{s}^{2\sigma}\right) \|x-y\|_0
 \lesssim_n \left( \|x\|_0^{2\sigma}+\|y\|_0^{2\sigma}\right) \|x-y\|_0 
 \le 2R^{2\sigma}\|x-y\|_0.
\end{align*}
Since $A_{|H_n}$ is a bounded operator, we get 
\begin{equation*}
\|h_n(x)-h_n(y)\|_0 \lesssim_{n,R}\|x-y\|_0.
\end{equation*}
From \ref{H3} we immediately see that $l_n$ is locally Lipschitz continuous  in $H_n$. 
This concludes the proof.
\end{proof}

Now we show that the local solution to \eqref{Galerkin} is actually a global solution. 

\begin{proposition}
\label{p_tight_1}
Assume  \eqref{sigma+n}, \ref{H1}, \ref{H2}(i), \ref{H3} and \ref{H5}. Then for any $n \in \mathbb{N}$ and  $u^0 \in H^{s}$   there exists 
a unique solution $u_n$ of \eqref{Galerkin} defined on the time interval $[0,+\infty)$ and with $\mathbb P$-a.e. path 
in $C([0,+\infty); H_n)$. 
 Moreover, for any $T$ and $\delta>0$, there exists $C_{\delta,T}>0$ such that 
\begin{equation}
\label{est_1}
\sup_{n \in \mathbb{N}}\mathbb{P} \left( \|u_n\|_{L^\infty(0,T;H^s)} \ge C_{\delta,T} \right) \le \delta.
\end{equation}
\end{proposition}

\begin{proof}
We follow the approach of \cite{BagMauXu2023}. We divide the proof in two parts: first we prove 
global existence and then the estimate \eqref{est_1}.

For any $n \in \mathbb{N}$ we endow the $H_n$ space with the $H^{s}$-scalar product.
We fix $n \in \mathbb{N}$ and take the unique maximal solution $(u_n, \tau_n)$ from Proposition \ref{local_sol}. We prove that the solution is global, that is $\tau_n=+ \infty$ $\mathbb{P}$-a.s., appealing to the Khasmiskii's test for non explosion, see  \cite[Theorem III.4.1]{Kah} (for the finite-dimensional case). 
The idea is as follows. We introduce a sequence $\{\tau_{n,k}\}_{k\in\mathbb{N}}$ of stopping times defined by
	\begin{align*}
	\tau_{n,k}:=\inf \left\{t\ge0: \|u_n(t)\|_s\ge k\right\},\qquad k\in\mathbb{N}.
	\end{align*}
In order to prove that $\tau_n=+ \infty$, $\mathbb{P}$-a.s., it is sufficient to find a Lyapunov function $\mathscr{V}:H^{s} \rightarrow \mathbb{R}$ satisfying
\begin{align}
 &\mathscr{V} \ge 0  \qquad \text{on} \quad H^{s}, \label{K1}\\
 &a_k:= \inf \bigl\{ \mathscr{V}(v): \|v\|_{s}\ge k \bigr\} \rightarrow \infty, \qquad \text{as} \quad k \rightarrow \infty, \label{K2}\\
 &\mathscr{V}(u_n(0))< \infty,  \label{K3}
\end{align}
such that
\begin{equation}
\label{K4}
\mathbb{E}[\mathscr{V}(u_n(t\wedge \tau_{n,k}))] \le \mathscr{V}(u_n(0))+ Ct,
\end{equation}
for a constant $C< \infty$ and all $t\in [0,T]$ and $k \in \mathbb{N}$.
Therefore,
\begin{equation*}
\mathbb{P}(\tau_{n,k}<t)
= \mathbb{E} \left(\mathbbm{1}_{\{\tau_{n,k}<t\}} \right) 
 \le \frac{1}{a_k}\mathbb{E}\left[\mathbbm{1}_{\{\tau_{n,k}<t\}}\mathscr{V}(u_n(t \wedge \tau_{n,k})) \right]
\le \frac{\mathscr{V}(u_n(0))+Ct}{a_k}.
\end{equation*}
Passing to the limit, we get
\begin{equation*}
\lim_{k \rightarrow \infty}\mathbb{P}(\tau_{n,k}<t) =0,
\end{equation*}
for every fixed $t\ge 0$. 
Therefore $\mathbb{P}(\tau_n<t)=\displaystyle\lim_{k \rightarrow \infty}\mathbb{P}(\tau_{n,k}<t)=0$ for every fixed $t \ge 0$, which means $\mathbb{P}(\tau_n=+\infty)=1$.

We consider $R>\frac 12 $, $a \in (0, \log(2R))$ and  a non-decreasing $C^2$-function 
$l:[0,+\infty)\to [a,+\infty)$,   such that
\begin{equation}
\label{elle1}
\begin{cases}
l(\rho)=a, & 0\le \rho<R\\
l(\rho)=\log_e \rho, &\rho>2R.
\end{cases}
\end{equation}
Then 
as Lyapunov function we consider 
\[
\mathscr{V}(u)=l(\|u\|_{s}).
\]
It is trivial to verify that $\mathscr{V}$ fulfills \eqref{K2} and \eqref{K3}.
In order to get \eqref{K4}, we apply  It\^o formula to $\mathscr{V}(u_n)$, up to the maximal existence time of the process $u_n$.  We obtain 
\begin{equation} \label{ITO_n}
{\rm d}\mathscr{V}(u_n(t))
=(\mathscr{L}_n\mathscr{V})(u_n(t)) \,{\rm d}t
    + \mathscr{V}^\prime(u_n(t))[P_n \phi(u_n(t))]\,{\rm d}W(t)  ,
\end{equation}
where $(\mathscr{L}_n\mathscr{V})(u_n(t))$  and $\mathscr{V}^\prime(u_n(t))$ vanish when 
$\|u_n(t)\|_{s}<R$; however, when $\|u_n(t)\|_{s}\ge R$ they are given by 
\begin{equation*}
 (\mathscr{L}_n\mathscr{V})(u_n)
 =
 \mathscr{V}^\prime(u_n)[\im A u_n - \im \alpha P_n(F(u_n))] + \frac 12 \mathscr{V}^{\prime \prime}(u_n)
[P_n \phi(u_n), P_n \phi(u_n)],
\end{equation*}
where 
\begin{equation} \label{v-primo}
\mathscr{V}^\prime(u_n)[h] = l^\prime(\|u_n\|_{s}) \frac{\Re\big( u_n, h\big)_{s}} {\|u_n\|_{s}}
\end{equation}
and 
\begin{equation*}
\mathscr{V}^{\prime \prime}(u_n)[h,k]=l^{\prime \prime}(\|u_n\|_{s}) \frac{\Re\big( u_n, h\big)_{s}
  \Re\big( u_n, k\big)_{s}} {\|u_n\|^2_{s}}
+ l^\prime (\|u_n\|_{s}) \left(  \frac{\Re\big( h, k\big)_{s}} {\|u_n\|_{s}} -  \frac{\Re\big( u_n, h\big)_{s}
   \Re\big( u_n, k\big)_{s}} {\|u_n\|^3_{s}}\right).
\end{equation*}
We notice the following simplification:
\[\begin{split}
 \mathscr{V}^\prime(u_n)[\im A u_n - \im \alpha P_n(F(u_n))] 
   &=\frac {l^\prime(\|u_n\|_{s})}{\|u_n\|_{s} }  
 \Re\big( u_n, \im A u_n - \im  \alpha P_n(F(u_n))\big)_{s}
 \\
 &= \frac {l^\prime(\|u_n\|_{s})}{\|u_n\|_{s} }  
 \Re\big( u_n, -\im \alpha P_n(F(u_n))\big)_{s}
  \\
 &= \frac {l^\prime(\|u_n\|_{s})}{\|u_n\|_{s} }  
 \Im\big( u_n, \alpha P_n(F(u_n))\big)_{s},
\end{split}\]
since  the operators $A$ and $(I+A)^{\frac s2}$ commute and so 
\begin{equation}\label{prodotto-reale}
\Re\big( u_n,\im A u_n\big)_s=0. 
\end{equation}
Hence 
\begin{equation}\label{Ito-dt}
\begin{split}
 (\mathscr{L}_n\mathscr{V})(u_n)
 &=
\frac {l^\prime(\|u_n\|_{s})}{\|u_n\|_{s} }  
 \Im \big( u_n, \alpha  P_n(F(u_n))\big)_{s}
 +\frac 12 l^{\prime\prime}(\|u_n\|_{s}) \frac{ [\Re \big(u_n,P_n \phi (u_n)\big)_{s}]^2}{\|u_n\|_{s}^2}
      \\
  &\qquad
  +\frac 12      l^{\prime}(\|u_n\|_{s})  \left( \frac{\|P_n \phi(u_n)\|_{s}^2}{\|u_n\|_{s}}
                          - \frac{ [\Re \big(u_n,P_n \phi (u_n)\big)_{s}]^2 } {\|u_n\|_{s}^3}\right) 
\\
&\le 
  |\alpha|  K   l^\prime(\|u_n\|_{s})    \|u_n\|_{L^\infty}^{2\sigma}   \|u_n\|_{s}   +\frac 12 l^{\prime\prime}(\|u_n\|_{s}) \frac{ [\Re \big(u_n,P_n \phi (u_n)\big)_{s}]^2}{\|u_n\|_{s}^2}
      \\
  &\qquad
  +\frac 12      l^{\prime}(\|u_n\|_{s})  \left( \frac{\|P_n \phi(u_n)\|_{s}^2}{\|u_n\|_{s}}
                          - \frac{ [\Re \big(u_n,P_n \phi (u_n)\big)_{s}]^2 } {\|u_n\|_{s}^3}\right),           
\end{split}
\end{equation}
where $K$ is the constant in the estimate \eqref{stimaF}.

We will now show that 
\begin{equation}
\sup_{n \in \mathbb{N}}  \sup_{u\in H_n} (\mathscr{L}_n\mathscr{V})(u_n) < \infty.
\end{equation}
Since $ (\mathscr{L}_n \mathscr{V})(u_n)=0$ when the $H^s$-norm of $u_n$ is smaller that $R$, 
we have to consider two cases: when the $H^s$-norm of $u_n$ is in $[R,2R]$ or in $(2R,+\infty)$.
\\
$\bullet$ If $R\le \|u_n\|_{s}\le 2R$,  from \eqref{Ito-dt} we estimate
\[\begin{split}
(\mathscr{L}_n \mathscr{V})(u_n)
& \le
  |\alpha| K   l^\prime(\|u_n\|_{s})    \|u_n\|_{L^\infty}^{2\sigma}   \|u_n\|_{s}   
   +\frac 12 l^{\prime\prime}(\|u_n\|_{s})\|P_n \phi (u_n)\|_{s}^2
  +     l^{\prime}(\|u_n\|_{s}) \frac{\|P_n \phi(u_n)\|_{s}^2}{\|u_n\|_{s}} 
\\
&\lesssim
    l^\prime(\|u_n\|_{s})    \|u_n\|_s^{2\sigma+1}
     +\frac 12 l^{\prime\prime}(\|u_n\|_{s})\| \phi (u_n)\|_{s}^2
  +     l^{\prime}(\|u_n\|_{s}) \frac{\|\phi(u_n)\|_{s}^2}{R}.
\end{split}\]
Since  $l^\prime(\|u_n\|_{s})$, $l^{\prime\prime}(\|u_n\|_{s})$ are  continuous, 
they are bounded when  $R \le \|u_n\|_{s}\le 2R$; 
also  $\|\phi(u_n)\|_{s}$ is bounded when  $R \le \|u_n\|_{s}\le 2R$  in virtue of Assumption \ref{H2}(i). Therefore
\[
\sup_{n \in \mathbb{N}}  \sup_{R\le \|u_n\|_{s}\le 2R } (\mathscr{L}_n\mathscr{V})(u_n) < \infty.
\]
\\
$\bullet$ If $ \|u_n\|_{s}> 2R$, 
then $l^\prime( \|u_n\|_{s})=\frac 1{ \|u_n\|_{s}}$ and 
$l^{\prime\prime}( \|u_n\|_{s})=-\frac 1{ \|u_n\|_{s}^2}$. 
Hence, from estimate \eqref{Ito-dt} we infer
\begin{align}
\label{crucial_es}
 (\mathscr{L}_n\mathscr{V})(u_n) 
 &\le
  |\alpha| K \|u_n\|_{L^\infty}^{2\sigma} 
 +
 \frac 12 \frac {\|P_n \phi(u_n)\|_{s}^2}  {\|u_n\|^2_{s}} 
 - \frac{ [\Re \big(u_n,P_n \phi (u_n)\big)_{s}]^2 } {\|u_n\|_{s}^4}.
\end{align}

Thanks to \eqref{stima-K-phi} in Assumption \ref{H5}, this quantity is finite if we choose $R=\frac r2$ ($r$ given in  \eqref{stima-K-phi}).
Indeed we use
$\|P_n \phi(u_n)\|_{s}\le \| \phi(u_n)\|_{s}$ 
and $ \big(u_n,P_n \phi (u_n)\big)_{s}= \big(u_n, \phi (u_n)\big)_{s}$,
so
\[
\sup_{n \in \mathbb{N}}  \sup_{\|u_n\|_{s}\ge 2R }(\mathscr{L}_n\mathscr{V})(u_n) < \infty.
\]

Combining all the above cases we conclude that $(\mathscr{L}_n\mathscr{V}(u_n))$ is  bounded in $H_n$, 
uniformly in $n \in \mathbb{N}$, that is  there must exist a positive constant $C$  - independent of $k$ and $n$ - 
such that 
\begin{equation}\label{cresce-con-C}
\mathscr{L}_n\mathscr{V} \le C \quad \text{ for any }n \in \mathbb{N}.
\end{equation} 
Hence
\begin{equation*}
\mathscr{V}(u_n(t\wedge \tau_{n,k})) \le \mathscr{V}(u_n(0))
+ Ct
   + \int_0^{t\wedge \tau_{n,k}}\mathscr{V'}(u_n(r))[ \phi(u_n(r))]\, {\rm d}W(r).
\end{equation*}
Taking the expected value on both sides of the above estimate we get \eqref{K4}, from which we infer that $\tau_n= + \infty$ $\mathbb{P}$-a.s., that is, the solution is defined at any time $t\ge 0$.

\bigskip
Let us now prove estimate \eqref{est_1}. Fix $n \in \mathbb{N}$ and define
  \[
 \tau_M^n:=\inf \{t \ge 0: \mathscr{V}(u_n(t)) \ge M\}
 \]
 for $M >0$, to be chosen later on.
 We have $\left\{\displaystyle\sup_{t \in [0,T]}\mathscr{V}(u_n(t)) \ge M \right\}=\{\tau^{n}_M \le T\}$ and
$\mathbb{P}(\mathscr{V}(u_n(\tau^{n}_M)\ge M)=1$, so 
$\left\{ \displaystyle\sup_{t \in [0,T]}\mathscr{V}(u_n(t)) \ge M \right\}
=\{\tau^{n}_M \le T \}\cap \{\mathscr{V}(u_n( \tau_M^n)) \ge M\}
\subseteq 
 \{ \mathscr{V}(u_n(T \wedge \tau_M^n)) \ge M \}$.
 
 From \eqref{ITO_n} and \eqref{cresce-con-C} we have 
\[
\mathbb{E}[\mathscr{V}(u_n(t\wedge \tau^n_M))] 
\le 
\mathscr{V}(u_n(0))+ Ct\le  \mathscr{V}(u(0))+ Ct
\]
By the Markov inequality we infer
\begin{align*}
\mathbb{P} \big( \sup_{t \in [0,T]}\mathscr{V}(u_n(t)) \ge M\big)
\le \mathbb{P} \big( \mathscr{V}(u_n(T \wedge \tau_M^n)) \ge M\big)
\le \frac{ \mathscr{V}(u(0))+ Ct}{M}.
\end{align*}
Hence, for any $\delta>0$  we choose  $M\ge \max(\log_e(2R), \frac{ \mathscr{V}(u(0))+ CT}{\delta})$;
therefore there exists $C_{\delta,T}$ such that
\begin{equation} \label{stima-1-x-tightness}
\sup_{n \in \mathbb{N}}\mathbb{P} \left(\|u_n\|_{L^\infty(0,T;H^s)} \ge C_{\delta,T} \right) \le \delta.
\end{equation}
\end{proof}

\begin{remark}
\label{weaker_H3}
From the proof we see that the same result can be obtained if in Assumption \ref{H5} we assume the bound \eqref{stima-K-phi}
only on the set  $ \mathbb B^c_{r,H^s} \cap \mathbb B^c_{\rho, L^\infty}$ for some  $\rho>0$ and we add the condition
\begin{equation}
\label{stima_palla_infty}
\exists L>0: \; 
\|\phi(u)\|_s \le L \|u\|_s, \quad \forall \ u \in \mathbb B^c_{r, H^s} \cap \mathbb B_{\rho, L^\infty}.
\end{equation}

Indeed, using $\|P_n \phi(u_n)\|_{s}\le \| \phi(u_n)\|_{s}$ 
and $| \big(u_n,P_n \phi (u_n)\big)_{s}|= |\big(u_n, \phi (u_n)\big)_{s}|\le \|u_n\|_s\|\phi(u_n)\|_s$, 
 starting from \eqref{crucial_es} 
 we estimate on $\mathbb B^c_{r, H^s} \cap \mathbb B_{\rho, L^\infty}$
\begin{align}
( \mathscr{L}_n\mathscr{V})(u_n) 
 &\le
  |\alpha| K \|u_n\|_{L^\infty}^{2\sigma} 
 +
 \frac 12 \frac {\|P_n \phi(u_n)\|_{s}^2}  {\|u_n\|^2_{s}} 
 - \frac{ [\Re \big(u_n,P_n \phi (u_n)\big)_{s}]^2 } {\|u_n\|_{s}^4}
 \\&\le
  |\alpha| K \rho^{2\sigma}+\frac 32 L^2 .
\end{align}

\end{remark}

The bound \eqref{stima-1-x-tightness} is one of the two estimates appearing in the tightness criterium 
given in Proposition \ref{prop_tight}. 
Now we look for the other estimate, i.e. we provide a uniform estimate in probability 
for $u_n$ in $C^{0, \beta}([0,T];H^{-s-1})$ with $0<\beta<\frac 12$.

\begin{proposition}
\label{p_tight_2}
Assume  \eqref{sigma+n}, \ref{H1}, \ref{H2}(i), \ref{H3} and \ref{H5}. Let $u_n$ be the unique solution to equation \eqref{Galerkin} given in Proposition \ref{p_tight_1}. Let $0 < \beta < \frac 12$. For any $T>0$ and $\delta>0$ there exists a positive constant $C_{\delta,T}$ such that 
\begin{equation*}
\sup_{n \in \mathbb{N}} \mathbb{P}\left( \|u_n\|_{C^{0, \beta}([0,T];H^{-s-1})})\ge C_{\delta,T} \right)\le \delta.
\end{equation*}
\end{proposition}
\begin{proof}
Let $\delta>0$. Estimate \eqref{est_1} and the continuous embedding $H^{s} \hookrightarrow H^{-s-1}$ yield the existence of a constant $\tilde C_{\delta,T}>0$ such that 
\begin{equation*}
\sup_{n \in \mathbb{N}}\mathbb{P} \left(\|u_n\|_{L^\infty(0,T;H^{-s-1})} \ge \tilde C_{\delta,T} \right) \le \delta.
\end{equation*}
Thus, bearing in mind \eqref{norm_C}, it is sufficient to prove a uniform bound in probability for the semi-norm $[u_n]_{C^{0, \beta}(H^{-s-1})}$, defined in \eqref{norm_C}.
Let $M>0$, we introduce the stopping time 
\begin{equation}
\label{rho}
\rho_M^n:=\inf\{t \ge 0: \|u_n(t)\|_{s} \ge M\},
\end{equation}
and the event 
\begin{equation*}
A_M^n:=\left\{\sup_{t \in [0,T]}\|u_n(t)\|_{s} \ge M\right\}.
\end{equation*}
We have that $A_M^n \supseteq \{\rho^n_M \le T\}$. In virtue of \eqref{est_1} it also holds $\sup_{n \in \mathbb{N}}\mathbb{P}(A_M^n)< \delta$, provided that $M$ is sufficiently large. Hence, for such an $M$, it holds
\begin{align*}
\mathbb{P}\left([u_n]_{C^{0, \beta}(H^{-s-1})}>a\right)
&\le \mathbb{P} \left((A_M^n)^c \cap [u_n]_{C^{0, \beta}(H^{-s-1})}>a\right) + \mathbb{P}(A_M^n)
\\
& \le \mathbb{P}\left([u_n(\cdot \wedge \rho^n_M)]_{C^{0, \beta}(H^{-s-1})}>a\right)+ \delta 
\le \frac{1}{a^p} \mathbb{E} \left[[u_n(\cdot \wedge \rho^n_M)]^p_{C^{0, \beta}(H^{-s-1})} \right] + \delta,
\end{align*}
for some positive $a$ and $p$. Thus, if we show that 
\begin{equation}
\label{star5}
\sup_{n \in \mathbb{N}}\mathbb{E} \left[[u_n(\cdot \wedge \rho^n_M)]^p_{C^{0, \beta}(H^{-s-1})} \right] <\infty,
\end{equation}
then, for suitable $a_\delta>0$, we can infer
\begin{equation*}
\mathbb{P}\left([u_n]_{C^{0, \beta}(H^{-s-1})}>a_\delta \right) \le 2 \delta,
\end{equation*}
which will conclude the proof.
Let us thus prove \eqref{star5}. For $0 \le t_1 < t_2 \le T$, we have 
\begin{align*}
\|u_n(t_2 \wedge \rho_M^n)-u_n(t_1 \wedge \rho_M^n)\|^p_{-s-1}
&\lesssim_p
\left(\int_{t_1}^{t_2}\left\Vert  \im A u_n(r \wedge \rho_M^n)\right\Vert_{-s-1}  {\rm d}r \right)^p
   + \left( \int_{t_1}^{t_2}\left\Vert  \im P_n(F( u_n(r \wedge \rho_M^n)))\right\Vert_{-s-1} {\rm d}r \right)^p
\\
& \qquad +\left\Vert \int_{t_1\wedge \rho_M^n}^{t_2\wedge \rho_M^n} \phi( u_n(r))\, {\rm d}W(r) \right\Vert^p_{-s-1}.
\end{align*}
Set $I_1:=\left(\int_{t_1}^{t_2}\left\Vert  \im A u_n(r \wedge \rho_M^n)\right\Vert_{-s-1}  {\rm d}r \right)^p$. We estimate 
\begin{align*}
I_1&\le |t_2-t_1|^p \sup_{r \in [0,T]}\left\Vert A u_n(r \wedge \rho_M^n)\right\Vert^p_{-s-1}
\lesssim  |t_2-t_1|^p \sup_{r \in [0,T]}\left\Vert u_n(r \wedge \rho_M^n)\right\Vert^p_{1-s}
\\
&\lesssim|t_2-t_1|^p \sup_{r \in [0,T]}\left\Vert u_n(r \wedge \rho_M^n)\right\Vert^p_{s},
\end{align*}
where the last inequality is a consequence of the continuous embedding $H^s\subset H^{1-s}$, since we assumed \eqref{STAR}.
Hence, \eqref{rho} yields the existence of $M_1=M_1(M)>0$ such that $I_1 \le M_1|t_2-t_1|^p$ for all $n \in \mathbb{N}$ and $\mathbb{P}$-a.s..

Set $I_2:=  \left( \int_{t_1}^{t_2}\left\Vert  \im P_n(F( u_n(r \wedge \rho_M^n)))\right\Vert_{-s-1} {\rm d}r \right)^p$. Exploiting the continuous embedding $H^s \hookrightarrow H^{-s-1}$, \eqref{bound_H^s_Pn}, \eqref{stimaF} and the continuous embedding $H^{s} \subset L^\infty$, we estimate 
\begin{align*}
I_2&\le |t_2-t_1|^p \sup_{r \in [0,T]}\left\Vert P_n(F(u_n(r \wedge \rho_M^n))\right\Vert^p_{-s-1}
\lesssim  |t_2-t_1|^p \sup_{r \in [0,T]}\left\Vert P_n(F(u_n(r \wedge \rho_M^n))\right\Vert^p_{s}
\\
&\lesssim  |t_2-t_1|^p \sup_{r \in [0,T]}\| u_n(r \wedge \rho_M^n)\|^{(2\sigma +1)p}_{s}.
\end{align*}
Hence, \eqref{rho} yields the existence of $M_2=M_2(\sigma, p,M)>0$ 
such that $I_2 \le M_2|t_2-t_1|^p$ for all $n \in \mathbb{N}$ and $\mathbb{P}$-a.s.

Set $I_3:=\left\Vert  \int_{t_1\wedge \rho_M^n}^{t_2\wedge \rho_M^n} \phi( u_n(r))\, {\rm d}W(r) \right\Vert^p_{-s-1}$. 
To estimate this term we introduce the process
\[
a_n(t):=\pmb{1}_{t_1\wedge \rho_M^n \le t \le t_2\wedge \rho_M^n}
\]
and we notice that 
\[
 \int_{t_1\wedge \rho_M^n}^{t_2\wedge \rho_M^n} \phi( u_n(r))\, {\rm d}W(r)= \int_{t_1}^{t_2} a_n(r)\phi( u_n(r))\, {\rm d}W(r).
\]
For $p>1$, the Burkholder-Davis-Gundy inequality, \eqref{P_nU'} and the embedding $H^s \hookrightarrow H^{-s-1}$ yield 
\begin{align*}
\mathbb{E}[I_3] 
&=\mathbb{E} \left[ \left\Vert \int_{t_1}^{t_2} a_n(r)\phi( u_n(r))\, {\rm d}W(r) \right\Vert^p_{-s-1}\right]
\lesssim_p \mathbb{E} \left[\left( \int_{t_1}^{t_2}\|a_n(r)\phi(u_n(r))\|^2_{-s-1}\, {\rm d}r \right)^{\frac p2}\right] 
\\
&\le \mathbb{E} \left[\left( \int_{t_1}^{t_2}\|\phi(u_n(r \wedge \rho_M^n))\|^2_{-s-1}\, {\rm d}r \right)^{\frac p2}\right]
\\
&\le |t_2-t_1|^{\frac p2} \mathbb{E} \left[\sup_{r \in [0,T]}\| \phi(u_n(r \wedge \rho_M^n))\|^p_{-s-1}\right]
\lesssim |t_2-t_1|^{\frac p2} \mathbb{E} \left[\sup_{r \in [0,T]}\| \phi(u_n(r \wedge \rho_M^n))\|^p_{s}\right].
\end{align*}
In virtue of  Assumption \ref{H2}, we infer the existence of $M_3=M_3(\sigma, \beta, p,M)>0$ such that $ \mathbb{E}\left[ I_3 \right]\le M_3|t_2-t_1|^{\frac p2}$ for all $n \in \mathbb{N}$.
\\
Putting together the above estimates, we thus infer
\begin{equation}
\label{star6}
\mathbb{E}\left[\|u_n(t_2 \wedge \rho_M^n)-u_n(t_1 \wedge \rho_M^n)\|^p_{-s-1}\right] \lesssim_{p,M,\sigma,\beta,T}|t_2-t_1|^{\frac p2}.
\end{equation}
At this point we recall that for $1<p< \infty$ and $\beta <\gamma-\frac1p$ we have the Sobolev embedding $W^{\gamma,p}(0,T) \subset C^{0, \beta}([0,T])$. In particular, see \cite[Theorem B.1.5]{DPZ96}, for a continuous function $\varphi:[0,T] \rightarrow H^{-s-1}$ and $p, \beta, \gamma$ as above, we have 
\begin{equation}
\label{sob}
[\varphi]^p_{C^{0, \beta}([0,T];H^{-s-1})} \lesssim [\varphi]^p_{W^{\gamma,p}(0,T;H^{-s-1})}:=\int_0^T \int_0^T \frac{\|\varphi(t_2)-\varphi(t_1)\|^p_{H^{-s-1}}}{|t_2-t_1|^{1+\gamma p}}\, {\rm d}t_1\, {\rm d}t_2.
\end{equation}
Thus, from \eqref{star6} and \eqref{sob} we infer
\begin{align*}
\mathbb{E} \left[[u_n(\cdot \wedge \rho^n_M)]^p_{C^{0, \beta}(H^{-s-1})} \right] 
&\lesssim \mathbb{E} \left[[u_n(\cdot \wedge \rho^n_M)]^p_{W^{\gamma,p}(H^{-s-1})} \right] 
=\int_0^T \int_0^T \frac{\mathbb{E}\left[\|u_n(t_2 \wedge \rho_M^n)-u_n(t_1 \wedge \rho_M^n)\|^p_{-s-1}\right] }{|t_2-t_1|^{1+\gamma p}}\, {\rm d}t_1\, {\rm d}t_2
\\
& \lesssim_{p,M, \sigma, \gamma, T}\int_0^T \int_0^T |t_2-t_1|^{\left( \frac 12-\gamma\right)p-1}\,  {\rm d}t_1\, {\rm d}t_2.
\end{align*}
The  double integral is finite when $ (\frac 12-\gamma)p-1>-1$. 
Hence the right-hand side of the above inequality is uniformly bounded w.r.t. $n \in \mathbb{N}$ when $\gamma<\frac 12$. Indeed, given   $\gamma<\frac 12$ there exist $\beta\in (0,\frac 12)$ and $p\in (1,\infty)$ such that $\beta < \gamma-\frac 1p$. This proves \eqref{star5} and concludes the proof.
\end{proof}

\subsection{Existence of a global martingale solution}
\label{suc_sec_con}

In this section we will construct a global-in-time martingale solution of equation \eqref{NLS_abs}  as a limit of the  Galerkin  sequence $u_n$ as $n \to +\infty$. For arbitrary  $T>0$ we will work on a finite time interval $[0,T]$, for arbitrary $T>0$.
Let us recall from Section \ref{tight_sec} the definition of the set
\[
Z_T:=C([0,T];H^{s'}) \cap C_w([0,T];H^s).
\]
Proposition \ref{p_tight_1} and \ref{p_tight_2} provide the a priori estimates on the Galerkin approximating sequence: its law is tight in $Z_T$ thanks to Proposition \ref{prop_tight}.
In metric spaces, one can apply Prokhorov Theorem and Skorohod Theorem to obtain convergence from tightness. 
Even if we consider a subsequence, we denote it again by $u_n$ for simplicity of notation.
Since the space $Z_T$ is a locally convex space, one uses the Jakubowski generalization to non-metric spaces (see  \cite{Jak86}, \cite{Jak98}). This result allows to infer the existence of a probability space $(\widetilde{\Omega}, \widetilde{F}, \widetilde{\mathbb{P}})$ and $Z_T \times C([0,T];\mathbb{R})$-valued random variables $\{(\tilde{u}_n, \widetilde{W}_n)\}_{n \in \mathbb{N}}$, $(\tilde{u}, \widetilde{W})$, living in this probability space, such that the distribution of $\{(\tilde{u}_n, \widetilde{W}_n)\}_{n \in \mathbb{N}}$ is the same as $\{(u_n, W_n)\}_{n \in \mathbb{N}}$ for every $n \in \mathbb{N}$, and 
\begin{equation}
\label{converge}
\text{the sequence $\{(\tilde{u}_n, \widetilde{W}_n)\}_{n \in \mathbb{N}}$ converges to $(\tilde{u}, \widetilde{W})$ in $Z_T \times C([0,T];\mathbb{R})$, $\widetilde{\mathbb{P}}$-a.s.}
\end{equation}
\\
We denote the filtration generated by $\tilde u, \widetilde{W}$, and $\widetilde{\mathbb{P}}$-null sets by $\{\widetilde{\mathcal{G}}_t\}_{t \in [0,T]}$ and we define $\widetilde{\mathcal{F}}_t := \bigcap_{s>t} \widetilde{\mathcal{G}}_s$, a filtration satisfying the standard assumptions. The filtrations $\{\widetilde{\mathcal{G}}^n_t\}_{t \in [0,T]}$ and $\{\widetilde{\mathcal{F}}^n_t\}_{t \in [0,T]}$ for $\tilde{u}_n$ and $\widetilde{W}^n$ are defined similarly.
\\
For every $n \in \mathbb{N}$, $\widetilde{W}^n$ is an $\{\widetilde{\mathcal{F}}_t^n\}_{t \in [0,T]}$-adapted Brownian motion, and the limit process $\widetilde W$ is an $\{ \widetilde{\mathcal{F}}_t^n\}_{t \in [0,T]}$-adapted Brownian motion
 (see \cite[Lemma 4.1]{BagMauXu2023}).
\\
 Moreover, arguing as in \cite[Lemma 4.2]{BagMauXu2023}, one proves that, for any $n \in \mathbb{N}$, $(\widetilde{\Omega}, \widetilde{\mathcal{F}},\{\widetilde{\mathcal{F}}_t^n\}_t, \widetilde{\mathbb{P}}, \tilde{u}_n, \widetilde{W}_n)$ is a solution to \eqref{Galerkin}.

The next result shows how the convergence of $\tilde{u}_n$ to $\tilde u$ in $Z_T$ can be used for the convergence of the terms that appear in the Galerkin equation \eqref{Galerkin}.
\begin{lemma}
\label{limit}
Let $s$ and $s'$  be as in \eqref{STAR}. Assume  \eqref{sigma+n}, \ref{H1}--\ref{H3} and \ref{H5}.
For all $t \in[0,T]$ and  $\tilde{\mathbb{P}}$-a.s. 
\begin{itemize}
\item [(i)] 
$\displaystyle \lim_{n \rightarrow \infty}\tilde{u}_n(t)=\tilde{u}(t)\quad$ in $H^{s'}$, 
\item [(ii)]  
$\displaystyle \lim_{n \rightarrow \infty}\int_0^t A \tilde{u}_n(r ) {\rm d}r=\int_0^t A \tilde{u}(r ) {\rm d}r,
\quad$ in $H^{s'-2}$, 
\item [(iii)]  
$\displaystyle \lim_{n \rightarrow \infty}\int_0^t  P_n  F(\tilde{u}_n(r))\, {\rm d}r=\int_0^t  F(\tilde{u}(r))\, {\rm d}r,
\quad$ in $H^{s'}$.
\end{itemize}
\begin{itemize}
\item [(iv)] 
$\displaystyle \lim_{n \rightarrow \infty}\int_0^t \|P_n \phi (\tilde{u}_n(r))-  \phi (\tilde{u}(r))\|^2_{-s-1}\, {\rm d}r=0$.
\end{itemize}
\end{lemma}
\begin{proof}
We know from \eqref{converge} that $\tilde u_n \rightarrow \tilde u$ in $Z_T$, $\widetilde{\mathbb{P}}$-a.s.
In particular, $\tilde u_n \rightarrow \tilde u$ in $C([0,T];H^{s'})$ $\widetilde{\mathbb{P}}$-a.s.. 
This implies i). Also ii) follows easily
\[
\left\Vert \int_0^t A \tilde u_n(r) {\rm d}r-\int_0^t A \tilde u(r) {\rm d}r\right\Vert_{s'-2}
\le
 \int_0^t \|A \tilde u_n(r)  - A \tilde u(r) \|_{s'-2} {\rm d}r
 \le t \sup_{0\le r\le t } \| \tilde u_n(r)- \tilde u(r)\|_{s'}.
\]
For iii):
\[\begin{split}
\left\Vert \int_0^t P_n F(\tilde u_n(r)) {\rm d}r-\int_0^t F(\tilde u(r)) {\rm d}r\right\Vert_{s'}
&\le
 \int_0^t \| P_n F(\tilde u_n(r)) - P_n  F(\tilde u(r))  \|_{s'}  {\rm d}r
 +
  \int_0^t \|  P_n F(\tilde u(r))-F(\tilde u(r))  \|_{s'}  {\rm d}r.
\end{split}
\]
We estimate each addend in the right hand side:
\[\begin{split}
 \int_0^t \| P_n F(\tilde u_n(r)) - P_n F(\tilde u (r))  \|_{s'}  {\rm d}r
& \le  \int_0^t \|  F(\tilde u_n(r)) -  F(\tilde u (r))  \|_{s'}  {\rm d}r
\\&  \lesssim_{s',d,\sigma}
 \int_0^t  \left( \|\tilde u_n(r)\|_{s'}^{2\sigma}+  \|\tilde u(r)\|_{s'}^{2\sigma}\right)
                \|(\tilde u_n(r)-\tilde u(r)\|_{s'}  {\rm d}r
                \\
 &   \lesssim
 t         \left(  \sup_{0\le r\le t }  \|(\tilde u_n(r)\|_{s'}^{2\sigma}
     +   \sup_{0\le r\le t }  \|\tilde u(r)\|_{s'}^{2\sigma}\right)
 \sup_{0\le r\le t }    \|(\tilde u_n(r)-\tilde u(r)\|_{s'}
 \end{split}
\]
and
\[\begin{split}
 \int_0^t \| P_n F(\tilde u(r))-F(\tilde u(r))  \|_{s'}  {\rm d}r
 & \le  t \|P_n-I\|_{\mathcal L(H^{s'})} \sup_{0\le r\le t } \| F(\tilde u(r)) \|_{s'} 
 \\
& \lesssim
 t \|P_n-I\|_{\mathcal L(H^{s'})} \sup_{0\le r\le t }   \| \tilde u(r) \|_{s'}^{2\sigma+1}, \quad
\text{ by Lemma  } \ref{Lemma_F}
  \end{split}               
\]
Both terms converge to zero as $n \rightarrow \infty$ in virtue of \eqref{limit_H^s} and since $\tilde u_n \rightarrow \tilde u$ in $Z_T$.
\\
It remains to show the convergence of the stochastic term.
To prove assertion (iv) we exploit Assumption \ref{H2}(ii). We write 
\begin{align*}
\int_0^t \|P_n \phi (\tilde{u}_n(r))-  \phi (\tilde{u}(r))\|^2_{-s-1}\, {\rm d}r
& \le \int_0^t \|P_n \phi (\tilde{u}_n(r))-  P_n\phi (\tilde{u}(r))\|^2_{-s-1}\, {\rm d}r + \int_0^t \|P_n \phi (\tilde{u}(r))-  \phi (\tilde{u}(r))\|^2_{-s-1}\, {\rm d}r.
\end{align*}
We estimate each addend in the right hand side.
\begin{align*}
 \int_0^t \|P_n \phi (\tilde{u}_n(r))-  P_n\phi (\tilde{u}(r))\|^2_{-s-1}\, {\rm d}r \le t \|P_n\|_{\mathcal{L}(H^{-s-1})} \sup_{0 \le r \le t} \|\phi (\tilde{u}_n(r))-\phi(u(r))\|^2_{-s-1},
\end{align*}
which converges to zero, as $n \rightarrow \infty$, since $\phi$ is a continuous map from $H^{s'}$ into $H^{-s-1}$ in virtue of Assumption \ref{H2}(ii) and $\tilde u_n \rightarrow \tilde u$ in $Z_T$.
\begin{align*}
\int_0^t \|P_n \phi (\tilde{u}(r))-  \phi (\tilde{u}(r))\|^2_{-s-1}\, {\rm d}r \le t \|P_n-I\|_{\mathcal{L}(H^{-s-1})}\sup_{0 \le r \le t}\|\phi(\tilde{u}(r))\|_{-s-1},
\end{align*}
which converges to zero, as $n \rightarrow \infty$ thanks to \ref{H2}(ii).
\end{proof}

\begin{remark}
\label{Hs'_rem}
From the proof of statement (iv) in Proposition \ref{limit} it is clear the role played by the auxiliary space $H^{s'}$ that appears in Assumption \ref{H2}. From the tightness in the space $Z_T$ we infer the strong convergence in $H^{s'}$and the weak convergence in $H^s$: this last convergence is not enough  to pass to the limit in the stochastic term.
\end{remark}

We have now all the ingredients to prove the following.
\begin{proposition}
\label{prop_ex_mar}
Assume  \eqref{sigma+n}, \ref{H1}-\ref{H3} and \ref{H5}, and let $s$ and $s'$ be as in \eqref{STAR}. 
Then for  every   $u^0 \in H^{s}$ there exists a  martingale solution to \eqref{NLS} defined  on the time interval  $[0,+\infty)$,
 with $\widetilde{\mathbb{P}}$-a.a. paths in $C([0,+\infty);H^{s'}) \cap C_w([0,+\infty);H^{s})$.
\end{proposition}
\begin{proof}
We prove the existence of a solution on any finite time interval $[0,T]$, for arbitrary  $T>0$.
As said at the beginning of this section, we have proved that on a common probability space $(\widetilde{\Omega}, \widetilde{\mathcal{F}}, \widetilde{\mathbb{P}})$ there exist random variables $\{(\tilde{u}_n, \widetilde{W}_n)\}_{n \in \mathbb{N}}$ and $(\tilde u, \widetilde{W})$ such that, $\widetilde{\mathbb{P}}$-a.s. $(\tilde{u}_n, \widetilde{W}_n)$ converges to $(\tilde u, \widetilde{W})$ in $Z_T \times C([0,T]; \mathbb{R})$ and satisfies
\begin{equation}
\label{eq_P}
\tilde{u}_n(t)=\tilde{u}_n(0)+ \im  \int_0^tA  \tilde{u}_n(s)\,{\rm d} s-\im  \alpha \int_0^t P_nF(\tilde{u}_n(s))\, {\rm d}s + \int_0^t P_n\phi(\tilde{u}_n(s))  \,{\rm d} \widetilde W^n(s), \qquad t \in [0,T].
\end{equation}
In view of Lemma \ref{limit}(iv) and \cite[Lemma 4.3]{BagMauXu2023}, passing to $n\to\infty$ along a subsequence, we have that 
\begin{equation*}
 \int_0^t P_n\phi(\tilde{u}_n(s))  \,{\rm d} \widetilde W^n(s) \rightarrow  \int_0^t\phi(\tilde{u}(s))  \,{\rm d} \widetilde W(s), \qquad \text{in} \ H^{-s-1}, \quad \widetilde{\mathbb{P}}-a.s.
\end{equation*}
Bearing in mind Lemma \ref{limit}(i)-(iii), passing to the $\widetilde{\mathbb{P}}$-a.s. limit in each term in the equation \eqref{eq_P}, we obtain, for any $t \in [0,T]$, 
\begin{equation*}
\tilde{u}(t)=\tilde{u}(0)+ \im \int_0^tA  \tilde{u}(s)\,{\rm d} s-\im \alpha\int_0^t F(\tilde{u}(s))\, {\rm d}s + \int_0^t\phi(\tilde{u}(s))  \,{\rm d} \widetilde W(s),
\end{equation*}
$\tilde{\mathbb{P}}$-a.s., 
understood as an identity in $H^{-s-1}$.
Thus $(\widetilde{\Omega}, \widetilde{\mathcal{F}},\widetilde{\mathbb{F}}, \widetilde{\mathbb{P}}, \tilde{u}, \widetilde{W})$ is a martingale solution to \eqref{NLS_abs} on $[0,T]$. This concludes the proof.
\end{proof}

We now exploit the mild formulation of equation \eqref{NLS_abs} to infer that the solution process has $\mathbb{\tilde{P}}$-a.s. paths in $C([0,\infty);H^s)$.
\begin{proposition}
\label{reg_Str_est}
Let $(\widetilde{\Omega}, \widetilde{\mathcal{F}},\widetilde{\mathbb{F}}, \widetilde{\mathbb{P}}, \tilde{u}, \widetilde{W})$ be the martingale solution to \eqref{NLS_abs},  given in Proposition \ref{prop_ex_mar}. 
Then, $\tilde u$ has $\tilde{\mathbb{P}}$-a.s. trajectories in $C([0,+\infty);H^s)$.
\end{proposition}
\begin{proof}
Let $(\widetilde{\Omega}, \widetilde{\mathcal{F}},\widetilde{\mathbb{F}}, \widetilde{\mathbb{P}}, \tilde{u}, \widetilde{W})$ be the martingale solution to \eqref{NLS_abs}, given in Proposition \ref{prop_ex_mar}. 
Let us choose and fix $t>0$.
We apply the It\^o formula, see \cite[Theorem 2.4]{BvNVW},  to the process $\im f(r, \tilde u(t-r))$, $r \in [0,t]$,  where $f$ is the function defined as
\begin{equation*}
f : [0,t]\times H^{-s-1} \ni  (r,x) \mapsto e^{-\im (t-r)A}x \in H^{-s-3}.
\end{equation*}
The function $f$ is of  $C^{1,2}$-class and 
we deduce that $\widetilde{\mathbb{P}}$-a.s.,
\begin{align}
\label{mild_form}
\tilde u(t)=e^{-itA}u^0 -\im \alpha \int_0^t e^{-i(t-r)A}F(\tilde{u}(r))\, {\rm d}r+\int_0^t e^{-i(t-r)A}\phi(\tilde{u}(r))\,{\rm d}\widetilde{W}(r),
\end{align}
in $H^{-s-3}$.
We now show that,  $\widetilde{\mathbb{P}}$-a.s., each addend in the right hand side is in $C([0,+\infty);H^s)$.
\\
Since $u^0\in H^s$, this is trivial for the first term. 
For the second term, it is enough to show that the integrand 
$[0,t]\ni r \mapsto e^{-i(t-r)A}F(\tilde{u}(r))\in H^s$ is integrable. From estimate \eqref{stimaF} and the continuous Sobolev embedding \eqref{Sobolev} we have
\[
\| e^{-i(t- r)A}F(\tilde{u}(r)) \|_{s}
\le
\|F(\tilde{u}(r)) \|_{s}
\le
K \|\tilde{u}(r)\|^{2\sigma}_{L^\infty} \|\tilde{u}(r)\|_{s}
 \lesssim_{s, \sigma, K}
  \|\tilde{u}(r)\|^{2\sigma+1}_{s},
\]
which is integrable on any finite time interval $[0,t]$, since  pathwise we have that 
$\tilde{u} \in C_w([0,t]; H^s) \subset L^\infty(0,t; H^s)$.

For the stochastic integral in \eqref{mild_form}  it is enough to show that the integrand in square integrable, $\mathbb{P}$-a.s.; 
in this way the stochastic integral exists as a local martingale and has a continuous version. 
We have
\[
\int_0^t \|e^{-i(t-r)A}\phi(\tilde{u}(r))\|_s^2 {\rm d}r
\le
\int_0^t \|\phi(\tilde{u}(r))\|_s^2 {\rm d}r
\]
which is  $\widetilde{\mathbb{P}}$-a.s. finite from  Assumption \ref{H2}(i), since   $\tilde{u} \in C_w([0,t]; H^s) \subset L^\infty(0,t; H^s)$.
This concludes the proof.
 \end{proof}

\subsection{Pathwise uniqueness}
\label{path_uniq_sec}
We now prove the  pathwise uniqueness of the martingale solutions.
\begin{proposition}
\label{path_uniq_prop}
Assume  \eqref{sigma+n} and \ref{H1}-\ref{H5}.
 Then, the pathwise uniqueness holds for equation \eqref{NLS_abs}, i.e., given any $T>0$, if $u_1$, $u_2$ are two 
martingale solutions to \eqref{NLS_abs} on the time interval $[0,T]$, defined on the same filtered probability space $(\Omega, \mathcal{F}, \mathbb{F}, \mathbb{P})$ with respect to the same Brownian motion $W$, with the same initial data in $H^s$, then 
\begin{equation*}
\mathbb{P}\left(u_2(t)=u_1(t), \quad \forall \ t \in [0,T]\right)=1.
\end{equation*}
\end{proposition}

\begin{proof}
Let $T>0$, we consider two martingale solutions on the time interval $[0,T]$,  defined on the same filtered probability space with the same Brownian motion and initial data in $H^s$. 
Given $M>0$ we consider the stopping times
\begin{equation*}
\tau_M^i:=\inf\{ t \in [0,T] \ : \|u_i(t)\|_{s} >M\}, \qquad i=1,2,
\end{equation*}
and set $\tau_M:= \tau_M^1 \wedge \tau_M^2$.
\\
Set  $v:=u_1-u_2$. Thanks to Lemma \ref{tec_lem} the equality 
\begin{equation}\label{Ito-unicita_bis}
\begin{split}
{\rm d} \|v(t)\|_{s}^2
&=2 \Re\big( v(t), -\im \alpha F(u_1(t))+ \im \alpha F(u_2(t))\big)_{s} {\rm d}t 
\\
&+ 2\Re\big( v(t), \phi(u_1(t))-\phi(u_2(t))\big)_{s} {\rm d}W(t)
+\|\phi(u_1(t))-\phi(u_2(t))\|^2_{s}\, {\rm d}t .
\end{split}
\end{equation}
is satisfied $\mathbb{P}$-a.s. on the time interval $[0, \tau_M]$.
\\
By means of the Cauchy-Schwartz inequality, \eqref{F3} and \eqref{STAR}, we estimate, within the interval $[0, \tau_M]$,
\begin{align*}
| \big( v(t), -\im\alpha F(u_1(t))+\im\alpha F(u_2(t))\big)_{s}|
&\le |\alpha|  \|v(t)\|_{s}  \|F(u_1(t))-F(u_2(t))\|_{s}
\\
& \lesssim  \left( \|u_1(t)\|^{2\sigma}_{s}+ \|u_2(t)\|^{2\sigma}_{s}\right) \|v(t)\|^2_{s} \lesssim C_M \|v(t)\|^2_{s},
\end{align*}
for a positive constant $C_M$.
Assumption \ref{H5} yields 
\begin{align*}
\|\phi (u_1(t))-\phi (u_2(t))\|_{s}
 \le \psi(\|u_1(t)\|_{s}, \|u_2(t)\|_{s} )\|u_1(t)-u_2(t)\|_{s}, 
\end{align*}
where $\psi:\mathbb{R}^+\times \mathbb{R}^+\rightarrow \mathbb{R}^+$  is a measurable locally bounded function.
Hence, for $t \in [0, \tau_M]$, there exists a positive constant $C_M$ such that
\begin{align*}
\|\phi (u_1(t))-\phi (u_2(t))\|^2_{s} \le C_M \|v(t)\|^2_{s}.
\end{align*}

Now we go back to \eqref{Ito-unicita_bis}, in its integral form. 
Taking into account the previous estimates, by applying the Burkholder-Davis-Gundy and the H\"older inequality to estimate the stochastic integral in \eqref{Ito-unicita_bis},  
we obtain
\begin{equation}\label{Ito-unicita_2}
\begin{split}
\mathbb{E}& \left[\sup_{r \in [0,t]}\|v(r \wedge \tau_M)\|^{4}_{s}\right]
\\
&\lesssim_{M,p} 
\mathbb{E}\left[\left( \int_0^{t \wedge \tau_M}\|v(r )\|^2_{s}  {\rm d}r \right)^2 \right]
+ 
\mathbb{E}\left[\left( \int_0^{t \wedge \tau_M}\|\phi(u_1(r))-\phi(u_2(r))\|^2_{s}\|v(r )\|^2_{s}\, {\rm d}r\right)\right]
\\
& \lesssim_{M,p} \mathbb{E}\left[\left( \int_0^{t \wedge \tau_M}\|v(r )\|^2_{s}\, {\rm d}r \right)^2 
\right] + \mathbb{E}\left[\left( \int_0^{t \wedge \tau_M}\|v(r )\|^4_{s}\, {\rm d}r \right) 
\right]
\\
& \lesssim_{M,p,T} \mathbb{E}\left[\int_0^{t \wedge \tau_M}\|v(r )\|^{4}_{s}\, {\rm d}r \right].
\end{split}\end{equation}
Since
\[
\mathbb{E}\left[\int_0^{t \wedge \tau_M}\|v(r)\|^{4}_{s}\, {\rm d}r\right]
\le \mathbb{E}\left[\int_0^t\|v(r \wedge \tau_M)\|^{4}_{s}\, {\rm d}r\right]
\le \int_0^t\mathbb{E}\left[\sup_{s \in [0,r]}\|v(s \wedge \tau_M)\|^{4}_{s}\right]\, {\rm d}r,
\]
from \eqref{Ito-unicita_2} we get
\begin{equation*}
\mathbb{E}\left[\sup_{r \in [0,t]}\|v(r \wedge \tau_M)\|^{4}_{s}\right]
\lesssim_{M,p,T}
\int_0^t\mathbb{E}\left[\sup_{r \in [0,\rho]}\|v(r \wedge \tau_M)\|^{4}_{s}\right]\, {\rm d}\rho .
\end{equation*}
By the Gronwall lemma we infer 
\begin{equation*}
\mathbb{E} \left[\sup_{r \in [0,t]}\|v(r \wedge \tau_M)\|^{4}_{s}\right]=0
\end{equation*}
for any $t \in [0,T]$; 
hence
\[
\sup_{r \in [0,T]}\|v(r \wedge \tau_M)\|^{4}_{s}=0 \qquad \mathbb{P}-a.s.
\]
Since both $u_1$ and $u_2$ live in $C([0,T];H^{s})$ $\mathbb{P}$-a.s., then $\tau_M^1 \rightarrow T$ and $\tau_M^2 \rightarrow T$, as $M \rightarrow \infty$, and
\[
\sup_{r \in [0,T]}\|v(r)\|^{4}_{s}=0 \qquad \mathbb{P}-a.s.
\]
 This concludes the proof.
\end{proof}



Existence of a martingale solution and pathwise uniqueness yield the existence of a unique strong solution (see e.g. \cite[Theorem 2]{Ondrejat_2004_Uniqueness} and \cite[Theorem 5.3 and Corollary 5.4]{Kunze_2013_Yamada}). Thus Theorem \ref{mainTH} follows as an immediate consequence of Propositions \ref{prop_ex_mar}, \ref{reg_Str_est} and \ref{path_uniq_prop}.

\section{Ergodic results }
\label{erg_res_sec}

In this section we are concerned with the study of the ergodic properties of the solution to equation \eqref{NLS_abs}. 

Following an idea of \cite{FlaGat},
in section \ref{mart_stat_sec} we prove that, under Assumptions \ref{H1}--\ref{H3} and \ref{H5bis}, we can construct stationary martingale solutions of equation \eqref{NLS_abs} as the limit of stationary solutions of the approximating finite-dimensional Galerkin system \eqref{Galerkin}. 
To prove the existence of at least one invariant measure, in section \ref{ex_inv_meas_sec},  we notice that, under assumptions \ref{H1}--\ref{H4} and \ref{H5bis}, the transition semigroup associated to \eqref{NLS_abs} is well defined and we readily have existence of invariant measures.
In section \ref{uni_inv_meas_sec} we work under the stronger Assumption \ref{H5bisbis} and we prove that $\mu=\delta_0$ is the unique invariant measure. Moreover, we prove the stability of the zero solution process. 

 \subsection{Stationary martingale solutions}
 \label{mart_stat_sec}

Let us  start with the definition of stationary martingale solutions. 
\begin{definition}[stationary martingale solution]\label{def-sta-martingale solution}

A \emph{stationary martingale solution} of the equation $\eqref{NLS_abs}$ on the time interval $[0,\infty)$ is a system $
\bigl(\tilde{\Omega},\tilde{\F},\tilde{\mathbb P},\tilde{\Filtration},\widetilde{W},u\bigr)
$
 consisting of
	\begin{itemize}
		\item a filtered probability space $\bigl(\tilde{\Omega},\tilde{\F},\tilde{\Prob},\tilde{\mathbb{F}})$,   satisfying the standard conditions; 
 		\item  a real valued one dimensional Brownian motion $\widetilde{W}$  on  $\bigl(\tilde{\Omega},\tilde{\F},\tilde{\Prob},\tilde{\Filtration}\bigr);$
		\item an 
  $\tilde{\Filtration}$-adapted process  $u$ with
$\tilde{\Prob}$-almost all paths in $C([0,\infty);H^s)$, 
	such that for every $t\in [0,T]$
equation \eqref{eqn-ItoFormSolution}
	holds $\tilde{\mathbb{P}}$-almost surely in $H^{-s-1}$. 
 This is a stationary process in $H^s$.
\end{itemize}
\end{definition}


 \begin{proposition}
\label{misura-invariante-per-Galerkin}
Assume  \eqref{sigma+n},  \ref{H1}, \ref{H2}(i), \ref{H3} and \ref{H5bis}.
 Then, for any $n \in \mathbb{N}$,  there exists 
a stationary solution $u_n^{st}$ of the Galerkin equation \eqref{Galerkin} such that, for any $t \ge 0$
\begin{equation}\label{stima-media-infty}
\sup_{n \in \mathbb{N}}\mathbb E \left[ \|u^{st}_n(t)\|_s^p \right] < \infty,
\end{equation}
where $p\in(0,1)$ is the parameter in condition \eqref{seconda-ipotesi-phi}.
Moreover, 
\begin{itemize}
\item for any $T$ and $\delta>0$, there exists $C_{\delta,T}>0$ such that 
\begin{equation}
\label{est_1_bis}
\sup_{n \in \mathbb{N}}\mathbb{P} \left(\|u_n^{st}(t)\|_{L^\infty(0,T;H^s)} \ge C_{\delta,T} \right) \le \delta,
\end{equation}
\item given $0 < \beta < \frac 12$, for any $T>0$ and $\delta>0$ there exists a positive constant $C_{\delta,T, \beta}$ such that 
\begin{equation}
\label{est_2_bis}
\sup_{n \in \mathbb{N}} \mathbb{P}\left( \|u_n^{st}\|_{C^{0, \beta}([0,T];H^{-s-1})}\ge C_{\delta,T, \beta} \right)\le \delta.
\end{equation}
\end{itemize}
\end{proposition}
\begin{proof}
We construct a stationary solution appealing to the Krilov-Bogoliubov's theorem, which provides an invariant measure $\mu_n$ for the Galerkin equation \eqref{Galerkin}. Let us consider the
Galerkin equation \eqref{Galerkin} with vanishing initial data $u_n(0)=0$. We denote by $\mathcal{L}(u_n(t;0))$ the law
of this solution process evaluated at time $t>0$. The law is supported on the finite-dimensional space $H_n$. We construct the sequence of time-averaged measures
\begin{equation*}
\frac1T \int_0^T \mathcal{L}(u_n(t;0))\, {\rm d}t, \quad T \in \mathbb{N}.
\end{equation*}
If we show that this sequence is tight in $H_n$, then we can infer the existence of a converging subsequence 
whose limit (as $T_k \rightarrow \infty$) is an invariant measure for the Galerkin equation (see e.g. \cite[Theorem 3.1.1]{DPZ96}).

To prove the tightness we introduce the Lyapunov function
\[
\mathscr{V}(u)=l(\|u\|_{s}),
\]
where $l:[0,+\infty)\to [a,+\infty)$  is a non-decreasing $C^2$-function  such that
\begin{equation}\label{elle2}
\begin{cases}
l(\rho)=a, & 0\le \rho<R\\
l(\rho)=\rho^p, &\rho>2R,
\end{cases}\end{equation}
for some $a \in (0, (r)^p)$ and choosing $R=\frac r2$.
The difference with respect to the function defined in \eqref{elle1} is for large values of $\rho$ (i.e. outside the ball of radius $2R$), where  we have 
\begin{equation}
\label{elle_prime}
l^\prime(\rho)= p \rho^{p-1} \quad 
\text{ and } \quad 
l^{\prime\prime}(\rho)= -p(1-p) \rho^{p-2}.
\end{equation}
Applying the It\^o formula to $\mathscr{V}(u_n)$ we obtain
\begin{equation}
\label{V_2}
{\rm d}\mathscr{V}(u_n(t))= (\mathscr{L}_n \mathscr{V})(u_n(t)) \,{\rm d}t
    + \mathscr{V}^\prime(u_n(t))[P_n \phi(u_n(t))]\,{\rm d}W(t)  .
\end{equation}
As in the proof of Proposition \ref{p_tight_1} we have 
that  $(\mathscr{L}_n \mathscr{V})(u_n) $ vanishes for $\|u_n\|_s<R$ and it is bounded when
$R\le  \|u_n\|_s \le 2R$;
so
\[
\sup_n \sup_{ \|u_n\|_s \le 2R} (\mathscr{L}_n \mathscr{V})(u_n) =:M<\infty.
\]
We have to check the estimates when $ \|u_n\|_{s}> 2R$.
If $ \|u_n\|_{s}> 2R=r$, from estimate \eqref{Ito-dt} and bearing in mind \eqref{elle_prime}, we infer
 \begin{equation}
\label{crucial_est_2}
 (\mathscr{L}_n \mathscr{V}) (u_n) 
 \le
 p  \|u_n\|^p_{s}\left(|\alpha |  K   \|u_n\|_{L^\infty}^{2\sigma}+\frac 12 \frac{\|P_n \phi(u_n)\|_{s}^2}{ \|u_n\|^2_s}
 -\frac{2-p}2\frac{ [\Re \big(u_n,P_n \phi (u_n)\big)_{s}]^2 } {\|u_n\|_{s}^4}\right).
\end{equation}
For $\|u_n\|_s>r$, \eqref{seconda-ipotesi-phi} in Assumption \ref{H5bis}  yields (setting $\tilde B=-B>0$)
\[
 (\mathscr{L}_n \mathscr{V}) (u_n) 
 \le
 -  p \tilde B  \|u_n\|_{s}^p.
\]
Summing up, 
\[
 (\mathscr{L}_n \mathscr{V}) (u_n) \le M -  p \tilde B  \|u_n\|_{s}^p 
 \qquad\forall \ u_n \in H^s
\]
and coming back to  \eqref{V_2} we have
\[
 {\rm d}\mathscr{V}(u_n(t))
 \le ( M -  p \tilde B  \|u_n(t)\|_{s}^p )  \,{\rm d}t
 + \mathscr{V}^\prime(u_n(t))[P_n \phi(u_n(t))]\,{\rm d}W(t).
\]
We now integrate in time; then we  take the expected value on both sides  (this requires to consider first a sequence of stopping times $\rho^n_M$ as in \eqref{rho} in order to deal with the stochastic integral which is  a  local martingale) and obtain
\begin{equation}
\label{KB}
\mathbb E\mathscr{V}(u_n(T)) +p \tilde B\int_0^T \mathbb E \|u_n(t)\|_{s}^p   \,{\rm d}t
\le
\mathscr{V}(u_n(0)) +MT,
\end{equation}
so,  bearing in mind that $\mathscr{V}$ is a non negative function and $\mathscr{V}(u_n(0))=\mathscr{V}(0)=a$,  for any $T\ge 1$ we have
\begin{equation}\label{stima-in-media-per-G}
\frac 1T \int_0^T \mathbb E( \|u_n(t)\|_{s}^p )  \,{\rm d}t
\le
\frac{a + M}{p \tilde B}.
\end{equation}
This provides  tightness  of the sequence $u_n$ in the  $H^s$-norm
and the Krylov-Bogoliubov technique applies 
(since the Markov semigroup associated to the Galerkin system is Feller) 
showing that there exists an invariant measure $\mu_n$ for any Galerkin approximation.

Now we consider the solution of the Galerkin equation with initial datum of law $\mu_n$; this is a stationary solution $u_n^{st}$ to equation \eqref{Galerkin}. 
From \eqref{stima-in-media-per-G} we get the estimate 
\begin{equation}\label{stima-media-mun}
\mathbb E ( \|u^{st}_n(t)\|_s^p ) \le \dfrac{a + M}{p \tilde B} \qquad \forall \ t \ge 0.
\end{equation}
This is a uniform estimate in $n \in \mathbb{N}$ and we infer \eqref{stima-media-infty}.

Estimate \eqref{est_1_bis} can be proved reasoning as in the proof of Proposition \ref{p_tight_1} but considering the Lyapunov function \eqref{elle2} instead of \eqref{elle1}.
Estimate \eqref{est_2_bis} is proved arguing exactly as in the proof of Proposition \ref{p_tight_2}. The only difference is that here we consider random initial data (see Remark \ref{rem_in_data}).
\end{proof}

By taking the limit of the stationary solutions of the Galerkin system \eqref{Galerkin} constructed in Proposition \ref{misura-invariante-per-Galerkin}, we construct  a stationary martingale solution to equation \eqref{NLS_abs}.
\begin{proposition}
\label{pro_sta_sol}
Assume  \eqref{sigma+n}, \ref{H1}, \ref{H2}, \ref{H3} and \ref{H5bis}. Then there exists a  stationary martingale solution 
$
\bigl(\tilde{\Omega},\tilde{\F},\tilde{\mathbb P},\tilde{\Filtration},\widetilde{W},\tilde{u}^{st}\bigr)
$ to \eqref{NLS_abs} with $\widetilde{\mathbb{P}}$-a.s. paths in $C([0,\infty);H^s)$ that satisfy the estimate 
\begin{equation}\label{stima-media-infty_u}
\mathbb E \left[ \|\tilde{u}^{st}(t)\|_s^p \right] < \infty, \quad t \ge 0,
\end{equation}
where $p\in(0,1)$ is the parameter in condition \eqref{seconda-ipotesi-phi}.
\end{proposition}
\begin{proof}
Let $u_n^{st}$ be the stationary solution of the Galerkin equation \eqref{Galerkin} given in Proposition \ref{misura-invariante-per-Galerkin}.
From the estimates \eqref{est_1_bis} and \eqref{est_2_bis} the sequence of laws $\{\mathcal{L}(u_n^{st})\}_{n \in \mathbb{N}}$ is tight in $Z_T$, for any $T>0$, thanks to Proposition \ref{prop_tight}. In view of Remark \ref{rem_inf} the sequence is tight in $Z_\infty$. Reasoning as in Section \ref{suc_sec_con} (see Proposition \ref{prop_ex_mar}), but bearing in mind that here we are dealing with random initial data (see Remark \ref{rem_in_data}), we construct a martingale solution to \eqref{NLS_abs}. This is a stationary process in $H^s$, in fact the weak and the strong Borel subsets of $H^s$ coincide. 
Therefore stationarity in $H^s$ is a consequence of the stationarity of the Galerkin sequence and  its 
convergence  in $C_w([0,+\infty);H^s)$. Reasoning as in the proof of Proposition \ref{reg_Str_est} one infers that $\widetilde{\mathbb{P}}$-a.s. $\tilde{u}^{st}$ has trajectories in $C([0,\infty);H^s)$.

The estimate \eqref{stima-media-infty_u} is inherited from the same estimate \eqref{stima-media-infty} for the Galerkin sequence.
\end{proof}

 \subsection{Existence of invariant measures}
 \label{ex_inv_meas_sec}

Under Assumptions \ref{H1}-\ref{H5bis} for any initial datum $x \in H^s$, there exists a unique global-in-time strong solution to \eqref{NLS_abs} with $\mathbb{P}$-a.a. paths in $C([0,\infty);H^s)$. This result follows from section
 \ref{S-wellposedness}, since Assumption \ref{H5bis} is stronger than Assumption \ref{H5}.

We denote the unique solution starting from $x$ by $u(t;x)$. We define the family of operators  $P:=\{P_t\}_{t\ge 0}$ associated to equation \eqref{NLS_abs} as 
\begin{equation}
\label{P_t}
(P_t\varphi)(x):= \mathbb{E}[ \varphi(u(t;x))], \quad x\in H^s,\quad \varphi\in \mathcal{B}_b(H^s),
\end{equation}
where $\mathcal{B}_b(H^s)$ is the space of Borel measurable bounded functions from $H^s$ to $\mathbb{R}$. 
$P_t\varphi$ is bounded 
for every $\varphi \in\mathcal{B}_b(H^s)$. We know from \cite[Cor.~23]{On05}
that the transition function is jointly measurable, that is for any $A\in\cB(H^s)$ 
the map $H^s \times [0,\infty)\ni  (x,t)\mapsto \mathbb{P}\{u(t;x)\in A\}\in \mathbb{R}$ is measurable. So $P_t\varphi$ is also measurable for every $\varphi \in \mathcal{B}_b(H^s)$, 
hence $P_t$ maps $\mathcal{B}_b(H^s)$ into itself for every $t \ge0$.
Furthermore, since the unique solution of \eqref{NLS_abs} is an $H^s$-valued continuous process, then it is also a Markov process, see \cite[Theorem 27]{On05}. 
Therefore, we deduce that the family of operators $\{P_t\}_{t\ge0}$ is a Markov semigroup, 
namely $P_{t+s}=P_tP_s$ for any $s,t\ge0$. 

\begin{definition}
  An invariant measure for the transition semigroup $P$ 
  is a probability measure on $(H^s, \cB(H^s))$ such that 
  \[
  \int_{H^s} \varphi(x)\,\mu({\rm d} x) = \int_{H^s} P_t\varphi(x)\,\mu({\rm d} x) \quad\forall\,t\geq0,\quad\forall\,\varphi\in \mathcal{B}_b(H^s).
  \]
\end{definition}

 \begin{theorem}
\label{misura-invariante}
Assume  \eqref{sigma+n} and \ref{H1}-\ref{H5bis}. 
Then there exists at least one invariant measure $\mu$  for equation \eqref{NLS_abs} such that 
\begin{equation}
\label{est_mu}
\int_{H^s} \|x\|_s^p {\rm d}\mu(x)<\infty,
\end{equation}
where $p\in (0,1)$ is the parameter in condition \eqref{seconda-ipotesi-phi}.
\end{theorem}

\begin{proof}
Since the transition semigroup $P$ given in \eqref{P_t} is well defined, the law at any given time of a stationary 
 martingale solution given in Proposition \ref{pro_sta_sol} is an invariant measure for equation \eqref{NLS_abs}. 
 Estimate \eqref{est_mu} follows from estimate \eqref{stima-media-infty_u}.
\end{proof}

\begin{remark} 
In the estimate \eqref{est_mu} the power $p$ is smaller than 1. This 
is different from the estimates of the $p$-moments of the $H^1$-norm, which can be obtained for larger $p$ when working in the energy space $H^1$ for the solution to the nonlinear Schr\"odinger equation with  an at most linear noise
 (see e.g. \cite{BFZ24}, \cite{FZ}).
\end{remark}

 \subsection{Uniqueness of the invariant measure and asymptotic stability of the zero solution}
\label{uni_inv_meas_sec}

Since Assumption  \ref{H5bisbis}  is stronger than  \ref{H5bis} and  \ref{H5}, from the previous sections we obtain the following result: 
under Assumptions \ref{H1}-\ref{H5bisbis} for any initial datum $x \in H^s$, there exists a unique global-in-time strong solution to \eqref{NLS_abs} with $\mathbb{P}$-a.s. paths in $C([0,\infty);H^s)$.
In addition, we now prove that 
all the solutions converge to $0$, that is the noise perturbation preserves the equilibrium solution $u=0$  of the NLS equation \eqref{NLS_det}. 
Moreover, we prove some asymptotic stability results for the zero solution in the spirit of \cite{BrzMasSei}. 
It then follows that  the invariant measure is unique and coincides with $\delta_0$.

\begin{theorem}
\label{uniq_inv_thm}
 Assume  \eqref{sigma+n} and \ref{H1}-\ref{H5bisbis}. Then for any $u^0\in H^s$ 
\begin{itemize}
    \item [i)] the zero solution to \eqref{NLS_abs} is exponentially stable in the $p$-mean, where $p\in (0,1)$ is the parameter in condition \eqref{terza-ipotesi-phi}, that is there exist constants $C< \infty$, $\lambda>0$, such that 
\begin{equation*}
        \mathbb{E}[ \|u(t)\|^p_s] \le C e^{-\lambda t}\|u^0\|^p_s, \qquad \forall \ t \ge 0;
\end{equation*}
\item [ii)] the zero solution to \eqref{NLS_abs} is exponentially stable with probability one, that is, for any $\bar \lambda \in (0,\lambda)$, there exists a $\mathbb{P}$-a.s. finite random time $\tau_0$ such that 
\begin{equation*}
\|u(t)\|^p_s \le C e^{-\bar \lambda t}\|u_0\|^p_s, \qquad \forall \ t \ge \tau_0, \quad \mathbb{P}-\text{a.s}.;
\end{equation*}
\item [iii)] $\mu=\delta_0$ is the unique invariant measure for equation \eqref{NLS_abs}.
\end{itemize}

\end{theorem}

The proof is based on the following auxiliary result.
\begin{lemma}
\label{lambda}
 Assume  \eqref{sigma+n} and \ref{H1}-\ref{H5bisbis}. 
Then there exists a constant $\lambda>0$ such that the process $\{e^{\lambda t}\|u(t)\|^p_s\}_{t \ge 0}$ is a non-negative continuous supermartingale; here $p\in (0,1)$ is the parameter in condition \eqref{terza-ipotesi-phi}.
\end{lemma}
\begin{proof}
    Let $u$ be the unique solution to \eqref{NLS_abs} starting from $u^0 \in H^s$. Since the function $x \mapsto \|x\|^p_s$ is not smooth enough for $p\in (0,1)$ to apply the It\^o formula we use an approximation. For  $\varepsilon >0$ we introduce the function 
\[
\mathscr{V}(x):=(\|x\|_{s}^2+\varepsilon)^{p/2},
\]
for $p\in (0,1)$ as in \eqref{terza-ipotesi-phi}. 
This is a smooth function and we compute 
\begin{equation} 
\mathscr{V}^\prime(x)[h] =p  (\|x \|_{s}^2+\varepsilon)^{\frac p2-1}\Re\big( x, h\big)_{s}
\end{equation}
and 
\begin{equation}
\mathscr{V}^{\prime \prime}(x)[h,k]=
p(p-2)( \|x\|_{s}^2+\varepsilon)^{\frac p2-2}\Re\big(x, h\big)_{s}\Re\big(x, k\big)_{s}
+p(\|x\|_{s}^2+\varepsilon)^{\frac p2-1} \Re\big( k, h\big)_{s}.
\end{equation}
We apply the It\^o formula to the process $X(r)=e^{\lambda r}\mathscr{V}(u(r))$ when $r \in [t_0,t]$ with $0\le t_0<t$.
We have 
\begin{equation}
\label{es_uni_1}
{\rm d}(e^{\lambda t}\mathscr{V}(u(t))) = \lambda e^{\lambda t}\mathscr{V}(u(t)) {\rm d}t+e^{\lambda t}{\rm d}\mathscr{V}(u(t)),
\end{equation}
where
\begin{equation} \label{ITO_n}
{\rm d}\mathscr{V}(u(t))
=(\mathscr{L}\mathscr{V})(u(t)) \,{\rm d}t
    + \mathscr{V}^\prime(u(t))[\phi(u(t))]\,{\rm d}W(t)  ,
\end{equation}
and  $(\mathscr{L}\mathscr{V})(u(t))$ is given by 
\begin{equation*}
 (\mathscr{L}\mathscr{V})(u(t))
 =
 \mathscr{V}^\prime(u(t))[\im A u(t) - \im \alpha (F(u(t)))] + \frac 12 \mathscr{V}^{\prime \prime}(u(t))
[\phi(u(t)),  \phi(u(t))].
\end{equation*}
Arguing in the proof of Proposition \ref{p_tight_1} and bearing in mind that $\phi(u)=f(u)u$  from Assumption \ref{H5bisbis},  we have that 
\[
\Re \big(u(t), \phi (u(t))\big)_{s}=\|u(t)\|_s^2\Re f(u(t))
\]
and
\[
\|\phi(u(t))\|_{s} =\|u(t)\|_s |f(u(t))|.
\]
We obtain the estimate
\begin{equation}
\label{es_uni_2}
\begin{split}
 (\mathscr{L}\mathscr{V})(u(t))
 &=
p  (\|u(t)\|_{s}^2+\varepsilon)^{\frac p2-1} 
 \Im \big( u(t), \alpha  (F(u(t)))\big)_{s}
 +\frac{p(p-2)}{2}\frac{\|u(t)\|_{s}^4}{ (\|u(t)\|_{s}^2+\varepsilon)^{2} } (\|u(t)\|_{s}^2+\varepsilon)^{\frac p2} [\Re f(u(t))]^2
 \\
 & \qquad + \frac p2 (\|u(t)\|_{s}^2+\varepsilon)^{\frac p2} \frac{\|u(t)\|_{s}^2}{ \|u(t)\|_{s}^2+\varepsilon  }|f(u(t))|^2
\\&\le 
p  (\|u(t)\|_{s}^2+\varepsilon)^{\frac p2}\left(  |\alpha|  K    \|u(t)\|_{L^\infty}^{2\sigma} \frac{\|u(t)\|_{s}^2}{\|u(t)\|_{s}^2+\varepsilon} -\frac{2-p}{2}\frac{\|u(t)\|_{s}^4}{ (\|u(t)\|_{s}^2+\varepsilon)^{2} } [\Re f(u(t))]^2 + \frac 12 \frac{\|u(t)\|_{s}^2}{ \|u(t)\|_{s}^2+\varepsilon  }|f(u(t))|^2\right),
\end{split}
\end{equation}
where $K$ is the constant in the estimate \eqref{stimaF}. Moreover, 
\begin{equation}
\label{es_uni_3}
    \mathscr{V}^\prime(u(t))[\phi(u(t))]\,{\rm d}W(t) = p  (\|u(t)\|_{s}^2+\varepsilon)^{\frac p2}\frac{\|u(t)\|_{s}^2}{\|u(t)\|_{s}^2+\varepsilon} [\Re f(u(t))]^2\,{\rm d}W(t).
\end{equation}
Coming back to estimate \eqref{es_uni_1}, bearing in mind \eqref{es_uni_2} and \eqref{es_uni_3} and using assumption
\ref{H5bisbis}, we obtain
\begin{align*}
    {\rm d}&(e^{\lambda t}\mathscr{V}(u(t))) \le  \lambda e^{\lambda t}\mathscr{V}(u(t)) {\rm d}t
        \\
        &\;+e^{\lambda t}
p  (\|u(t)\|_{s}^2+\varepsilon)^{\frac p2}  \frac{\|u(t)\|_{s}^2}{\|u(t)\|_{s}^2+\varepsilon} 
\left(  |\alpha|  K    \|u(t)\|_{L^\infty}^{2\sigma}  -\frac{2-p}{2}\frac{\|u(t)\|_{s}^2}{ \|u(t)\|_{s}^2+\varepsilon } [\Re f(u(t))]^2 + \frac 12  |f(u(t))|^2\right) {\rm d}t
\\
& \;+ e^{\lambda t} p  (\|u(t)\|_{s}^2+\varepsilon)^{\frac p2}\frac{\|u(t)\|_{s}^2}{\|u(t)\|_{s}^2+\varepsilon} [\Re f(u(t))]^2\,{\rm d}W(t)
\\
&\le  \lambda e^{\lambda t}\mathscr{V}(u(t)){\rm d}t
+e^{\lambda t}
\frac {p(2-p)}2  \varepsilon \frac{\|u(t)\|_{s}^2}{ (\|u(t)\|_{s}^2+\varepsilon)^{2-\frac p2}}  [\Re f(u(t))]^2{\rm d}t
+Bp e^{\lambda t} \frac{\|u(t)\|_{s}^2}{ (\|u(t)\|_{s}^2+\varepsilon)^{1-\frac p2}}{\rm d}t
\\
& \;+ e^{\lambda t} p  (\|u(t)\|_{s}^2+\varepsilon)^{\frac p2}\frac{\|u(t)\|_{s}^2}{\|u(t)\|_{s}^2+\varepsilon} [\Re f(u(t))]^2\,{\rm d}W(t).
\end{align*}
We now take the integral formulation of the above inequality (bearing in mind that we are working on the time interval $[t_0,t]$) and pass to the limit as $\varepsilon\rightarrow 0$. It holds 

\begin{equation*}
e^{\lambda t}\mathscr{V}(u(t)) \rightarrow e^{\lambda t}\|u(t)\|^p_s, \qquad \mathbb{P}-a.s.,
\end{equation*}

\begin{equation*}
    \int_{t_0}^t e^{\lambda r}\mathscr{V}(u(r))\, {\rm d}r \rightarrow \int_{t_0}^t e^{\lambda r}\|u(r)\|^p_s\,{\rm d}r, \qquad \mathbb{P}-a.s.,
\end{equation*}
\[
\varepsilon   \int_{t_0}^t   e^{\lambda r}  \frac{\|u(r)\|_{s}^2}{ (\|u(r)\|_{s}^2+\varepsilon)^{2-\frac p2}}  [\Re f(u(r))]^2 {\rm d}r \to 0 , \qquad \mathbb{P}-a.s.,
\]
and 
\[
  \int_{t_0}^t e^{\lambda r}  \frac{\|u(r)\|_{s}^2}{ (\|u(r)\|_{s}^2+\varepsilon)^{1-\frac p2}}  {\rm d}r
   \rightarrow 
 \int_{t_0}^t e^{\lambda r}\|u(r)\|_{s}^p
 {\rm d}r,\qquad \qquad \mathbb{P}-a.s..
\]
Moreover, we have the convergence in probability of the stochastic integral:
\begin{align*}
\int_{t_0}^t e^{\lambda r} (\|u(r)\|_{s}^2+\varepsilon)^{p/2}\frac{\|u(r)\|_{s}^2}{\|u(r)\|_{s}^2+\varepsilon} [\Re f(u(r))]^2\,{\rm d}W(r)
\rightarrow \int_{t_0}^t e^{\lambda r} \|u(r)\|_{s}^p[\Re f(u(r))]^2\,{\rm d}W(r)
\end{align*}
as a consequence of the fact that 
\begin{equation*}
    \int_{t_0}^t e^{2\lambda r}[\Re f(u(r))]^4\left| (\|u(r)\|_{s}^2+\varepsilon)^{p/2}\frac{\|u(r)\|_{s}^2}{\|u(r)\|_{s}^2+\varepsilon}  - \|u(r)\|_{s}^p \right|^2 \, {\rm d}r \rightarrow 0,
\end{equation*}
thanks to the Dominated Convergence Theorem (since $u \in L^\infty(0,T;H^s)$).

Thus, setting $\tilde B=-B>0$
 we obtain 
\begin{equation*}
e^{\lambda t}\|u(t)\|^p_s\le e^{\lambda t_0}\|u(t_0)\|^p_s + (\lambda - p \tilde B )\int_{t_0}^te^{\lambda r} \|u(r)\|_{s}^p + \int_{t_0}^te^{\lambda r} \|u(r)\|_{s}^p[\Re f(u(r))]^2\,{\rm d}W(r).
\end{equation*}
If we now choose $\lambda>0$ such that $\lambda-p \tilde B<0$ we obtain
\begin{equation}
\label{supermartingale}
e^{\lambda t}\|u(t)\|^p_s\le e^{\lambda t_0}\|u(t_0)\|^p_s + \int_{t_0}^te^{\lambda r} \|u(r)\|_{s}^p[\Re f(u(r))]^2\,{\rm d}W(r).
\end{equation}
The stochastic integral in the right hand side of the above inequality is a local martingale. In fact, if we define the stopping time $\tau_N:= \inf \{t>\tau :  \|u(t)\|_s>N\}$, with $N \in \mathbb{N}$, and 
\[
M_N(t):=\int_{t_0}^{t\wedge \tau_N}e^{\lambda r} \|u(r)\|_{s}^p[\Re f(u(r))]^2\,{\rm d}W(r),
\]
then 
\[
\mathbb{E}[|M_N(t)|^2] \le \int_{t_0}^{t\wedge \tau_N}e^{2\lambda r} \|u(r)\|_{s}^{2p}[\Re f(u(r))]^4\,{\rm d}r< \infty,
\]
since by the assumptions $f:H^s \rightarrow \mathbb{C}$ is bounded on balls and $u \in C([0,T];H^s)$. Hence, $M_N(t)$ is a square integrable martingale; in particular, $\mathbb{E}\left[M_N(t) |\mathcal{F}_{t_0}\right]=0$ for any $t \ge t_0$.
Therefore, taking the conditional expectation on both sides of \eqref{supermartingale}, up to the stopping time $\tau_N$, we get
\begin{equation*}
\mathbb{E}[e^{\lambda (t\wedge \tau_N)}\|u(t\wedge \tau_N)\|_s^p |\mathscr{F}_{t_0}] \le e^{\lambda t_0} \|u(t_0)\|_s^p  \qquad \forall  \ t_0<t.
\end{equation*}
Since $t \wedge \tau_N \rightarrow t$ $\mathbb{P}$-a.s. as $N \rightarrow \infty$, we get 
\begin{equation}
\label{eq4.}
\mathbb{E}[e^{\lambda t}\|u(t)\|_s^p |\mathscr{F}_{t_0}] \le e^{\lambda t_0} \|u(t_0)\|_s^p  \qquad \forall  \ t_0<t.
\end{equation}
which concludes the proof. \end{proof}

Let us now prove Theorem \ref{uniq_inv_thm}. Part of the proof is inspired by \cite[Proof of Theorem 1.4]{BrzMasSei}.
\begin{proof}[Proof of Theorem \ref{uniq_inv_thm}]
\begin{itemize}
\item [i)]
Let $u$ be the unique solution of problem \eqref{NLS} starting from $u^0 \in H^s$.
Lemma \ref{lambda} yields
\begin{equation}
\label{es1_delta0}
\mathbb{E}\left[ \|u(t)\|^p_s \right]\le e^{-\lambda t} \|u^0\|^p_s, \qquad \forall \ t \ge 0,
\end{equation}
which proves statement (i).
\item[ii)] Take an arbitrary $\bar \lambda \in (0, \lambda)$ and set $\varepsilon:= \lambda-\bar \lambda>0$. We have 
\begin{equation*}
   \mathbb{P}\left(\sup_{t \in [k,k+1]}e^{\bar \lambda t}\|u(t)\|^p_s\ge \|u^0\|^p_s \right) \le \mathbb{P}\left( \sup_{t \in [k,k+1]}e^{\lambda t}\|u(t)\|^p_s\ge e^{\varepsilon k}\|u^0\|^p_s\right),
\end{equation*}
for any $k \in \mathbb{N}$. Without loss of generality we assume $\|u^0\|_s>0$ (otherwise there is nothing to prove). By the Doob Supermartingale inequality, bearing in mind \eqref{es1_delta0}, we infer 
\begin{align*}
\mathbb{P}\left( \sup_{t \in [k,k+1]}e^{\lambda t}\|u(s)\|^p_s \ge e^{\varepsilon k}\|u^0\|^p_s\right) 
\le 
\frac{\mathbb{E}[e^{\lambda s}\|u(s)\|^p_s]}{e^{\varepsilon k}\|u^0\|^p_s} 
\le 
e^{-\varepsilon k}.    
\end{align*}
Thus 
\begin{equation*}
    \sum_{k=0}^\infty  \mathbb{P}\left(\sup_{t \in [k,k+1]}e^{\bar\lambda t}\|u(t)\|^p_s\ge \|u^0\|^p_s \right) < \infty,
\end{equation*}
hence statement (ii) follows from the Borel-Cantelli lemma.
\item[iii)]
 For the unique solution of problem \eqref{NLS}, we put in evidence the initial datum $u^0 \in H^s$
by writing $u(\cdot;u^0)$.
From (ii) we know that, for every $\bar \lambda \in (0, \lambda)$,
there exists a $\mathbb{P}$-a.s. finite random time  $\tau_0$ such that
\begin{equation*}
 \|u(t;u_0)\|^p_s \le e^{-\bar \lambda t} \|u^0\|^p_s, \qquad \forall \ t \ge \tau_0, \quad \mathbb{P}-a.s.
\end{equation*}
Hence
 \begin{equation}
 \label{lim_1}
\|u(t;u^0)\|_s \rightarrow 0, \quad \text{as} \ t \rightarrow \infty, \quad \mathbb{P}-a.s.
\end{equation}
Take any $\phi \in \mathcal{C}(H^s)$ i.e. $\phi:H^s \rightarrow \mathbb{R}$ continuous. From \eqref{lim_1} we infer 
\[
\phi(u(t;u^0)) \rightarrow \phi(0), \quad \text{as} \ t \rightarrow \infty,
\]
for any initial datum $u^0 \in H^s$. By the 
 Dominated Convergence Theorem we also get
 \begin{equation}
\label{lim_2}
(P_t\phi ) (u^0)=\mathbb E  \phi(u(t;u^0)) \rightarrow \phi(0), \quad \text{as} \ t \rightarrow \infty.
 \end{equation}
 Let now $\mu$ be any invariant measure. Then, from its definition, we have 
\[
  \int_{H^s} \phi(x)\,\mu({\rm d} x) = \int_{H^s} P_t\phi(x)\,\mu({\rm d} x) \quad\forall\,t\geq0,\quad\forall\,\phi\in \mathcal{B}_b(H^s).
\]
Taking $\phi \in \mathcal{C}_b(H^s)$, bearing in mind \eqref{lim_2}, 
by the Dominated Convergence Theorem the right hand side converges to $\phi(0)$ as $t \rightarrow \infty$. This implies
\begin{equation}
\label{final_inv}
    \int_{H^s}\phi(x)\, {\rm d}\mu(x)=\phi(0) \qquad\forall\,t\geq0,\quad\forall\,\phi\in \mathcal{C}_b(H^s).
\end{equation}
We notice that $\mathcal{C}_b(H^s)$ is a determining set for the measure. 
In fact, take $\phi(u)=e^{\im (u,h)_s}$, $h \in H^s$; with this choice the integral defines the characteristic function and this is enough to determine the measure. 
Thus from equality \eqref{final_inv} we conclude that $\mu=\delta_0$ is the unique invariant measure since \eqref{lim_2} holds for any initial datum $u^0$.
\end{itemize}
\end{proof}

As a consequence of Lemma \ref{lambda} we also infer stability in probability for the zero solution.
\begin{proposition}
 Under the same assumptions of Theorem \ref{uniq_inv_thm} the zero solution of equation \eqref{NLS_abs} is stable in probability, that is for every $\varepsilon>0$ there exists $\delta>0$ such that, for any solution $u$ starting from $u^0 \in H^s$,
 \begin{equation*}
\|u^0\|_s< \delta \qquad \Longrightarrow \qquad \mathbb{P}\left(\sup_{t \ge 0} \|u(t)\|_s>\varepsilon\right) < \varepsilon.
 \end{equation*}
\end{proposition}

\begin{proof}
Fix $\varepsilon>0$ and set
\begin{equation*}
\sigma_\varepsilon :=\inf\{t \ge 0:\quad  \|u(t)\|_s \ge \varepsilon\}. 
\end{equation*}
Lemma \ref{lambda} and the optional sampling theorem yields 
\begin{equation}
\label{sampling}
\mathbb{E}\left[\|u(t\wedge \sigma_\varepsilon)\|_s^p\right] \le \|u^0\|_s^p, \qquad \forall \ t \ge 0, 
\end{equation}
where we recall $p \in (0,1)$ is as in \eqref{terza-ipotesi-phi}.
 We have $=\{\sigma_\varepsilon < t\}=\left\{\displaystyle\sup_{r \in [0,t]}\|u(r)\|_s \ge \varepsilon \right\}$ and
$\mathbb{P}(\|u(\sigma_\varepsilon)\|\ge \varepsilon)=1$, so 
$\{\sigma_\varepsilon < t\}
=\{\sigma_\varepsilon < t \}\cap \{\|u(\sigma_\varepsilon)\|_s \ge \varepsilon\}
\subseteq 
 \{ \|u(t \wedge \sigma_\varepsilon)\|_s \ge \varepsilon \}$.
 Thus, thanks to the Markov inequality and estimate \eqref{sampling}, we infer
 \begin{equation*}
\mathbb{P}\left(\sigma_\varepsilon < t\right) \le \mathbb{P}\left(\|u(t \wedge \sigma_\varepsilon)\|_s \ge \varepsilon \right) \le \frac{\mathbb{E}\left[\|u(t\wedge \sigma_\varepsilon)\|_s^p\right]}{\varepsilon^p}\le \frac{\|u^0\|_s^p}{\varepsilon^p},
 \end{equation*}
for all $t \ge 0$. Since $\{\sigma_\varepsilon < t\} \nearrow \{\sigma_\varepsilon < \infty\}$ as $t \rightarrow \infty$, we obtain
 \begin{equation*}
\mathbb{P}\left(\sigma_\varepsilon < \infty\right) \le  \frac{\|u^0\|_s^p}{\varepsilon^p},
 \end{equation*}
 that is 
\begin{equation*}
\mathbb{P}\left( \sup_{t \ge 0}\|u(t)\|_s \ge \varepsilon\right)    \le  \frac{\|u^0\|_s^p}{\varepsilon^p}.
\end{equation*}
Taking $\delta>0$ such that $\delta=\varepsilon^{\frac{p+1}{p}}$, the thesis follows.
\end{proof}


\appendix

\section{Compactness and tightness results}
\label{tight_sec_main}
We recall here some deterministic compactness results and the tightness criteria.
We work with the triplet of spaces $H^s \subset H^{s'} \subset H^{-s-1}$, for $s$ and $s'$ as in \eqref{STAR}. Both embeddings are continuous and dense, moreover the embeddings $H^s \subset H^{s'}$ and $H^s \subset H^{-s-1}$ are compact.

\subsection{Deterministic compactness criteria}

Given $r>0$ let us consider the ball 
\begin{equation*}
\mathbb{B}_{r,H^s}:= \{x \in H^{s} \ : \ \|x\|_{s} \le r\}.
\end{equation*}
We simply write $\mathbb{B}$ for $\mathbb{B}_{r,H^{s}}$.  Let $\mathbb{B}_w$ denote the ball endowed with the weak topology  of $H^s$. Let us consider the following subspace of $C_w([0,T];H^{s})$
\begin{equation*}
C([0,T];\mathbb{B}_w)= \{u \in C_w([0,T];H^{s}) \ : \ \sup_{t \in [0,T]}\|u(t)\|_{s} \le r\}.
\end{equation*}
The space $C([0,T];\mathbb{B}_w)$ is metrizable w.r.t. a metric $\rho$ (see, e.g., \cite[Appendix A]{BHW2019}). Since, by the Banach-Alaoglu Theorem, $\mathbb{B}_w$ is compact, $(C([0,T];\mathbb{B}_w), \rho)$ is a complete metric space. Moreover, 
$u_n \rightarrow u$ in $C([0,T];\mathbb{B}_w)$ iff for any $h \in H^{-s}$
\begin{equation*}
\lim_{n \rightarrow \infty} \sup_{s \in [0,T]}|_{H^{s}}\langle u_n(s)-u(s), h\rangle_{H^{-s}}|=0.
\end{equation*}
We start with the following result.

\begin{lemma}
\label{technical_lemma}
Let $u_n: [0,T] \rightarrow H^{s}$, $n \in \mathbb{N}$, be functions s.t.
\begin{itemize}
\item [i)] $\displaystyle\sup_{n \in \mathbb{N}}\sup_{t \in [0,T]} \|u_n(t)\|_{s} \le r$, for some $r>0$,
\item [ii)] $u_n \rightarrow u$ in $C([0,T];H^{-s-1})$.
\end{itemize} 
Then $u, u_n \in C([0,T];\mathbb{B}_w)$, for all $n \in \mathbb{N}$, and $u_n \rightarrow u$ in $C([0,T];\mathbb{B}_w)$ as $n \rightarrow \infty$.
\end{lemma}
\begin{proof}
The proof follows the line of \cite[Lemma 2.1]{BM11} and \cite[Lemma 4.1]{BHW2019}.
From the assumptions and the Strauss Lemma \cite{Strauss}, we infer
\begin{equation*}
u_n \in L^\infty(0,T;H^{s}) \cap C_w([0,T]; H^{-s-1}) =C_w([0,T];H^{s})
\end{equation*}
and $\sup_{t \in [0,T]} \|u_n(t)\|_{s} \le r$, for any $n \in \mathbb{N}$. 
Moreover
\[
\|u\|_{L^\infty(0,T;H^{s})}\le \liminf_{n\to\infty} \|u_n\|_{L^\infty(0,T;H^{s})}.
\]
Hence we get $u_n \in C([0,T]; \mathbb{B}_w)$ for any $n \in \mathbb{N}$.
It remains to prove that $u_n \rightarrow u$ in $C([0,T];\mathbb{B}_w)$ as $n \rightarrow \infty$, that is, for all $h \in H^{-s}$ it holds
\begin{equation}
\label{star1}
\lim_{n \rightarrow \infty}|_{H^{s}}\langle u_n(s)-u(s), h\rangle_{H^{-s}}|=0, 
\end{equation}
and that $u \in C([0,T];\mathbb{B}_w)$.
\\
\textbf{Step 1.} First, let us fix $h \in H^{s+1}$. Then, 
\begin{align*}
\sup_{s \in [0,T]}|_{H^{s}}\langle u_n(s)-u(s), h\rangle_{H^{-s}}| 
=\sup_{s \in [0,T]}|_{H^{-s-1}}\langle u_n(s)-u(s), h\rangle_{H^{s+1}}| 
\le \|u_n-u\|_{C([0,T];H^{-s-1})}\|h\|_{s+1} \rightarrow 0, 
\end{align*}
as $n \rightarrow \infty$, by (ii). 
\\
\textbf{Step 2.} Now we show that \eqref{star1} holds for all $h \in H^{-s}$. Let $\delta>0$ and let $h \in H^{-s}$. Since $H^{s+1}$ is dense in $H^{-s}$, there exists $h_\delta \in H^{s+1}$ such that $\|h-h_\delta\|_{-s} \le \delta$. Using (i), we infer that, for all $s \in [0,T]$, the following estimate holds
\begin{align*}
|_{H^{s}}\langle u_n(s)-u(s), h\rangle_{H^{-s}}|
&\le |_{H^{s}}\langle u_n(s)-u(s), h-h_\delta\rangle_{H^{-s}}|+|_{H^{s}}\langle u_n(s)-u(s), h_\delta\rangle_{H^{-s}}|
\\
&\le \|u_n(s)-u(s)\|_{s}  \|h-h_\delta\|_{-s} + |_{H^{s}}\langle u_n(s)-u(s), h_\delta\rangle_{H^{-s}}|
\\
&\le \delta \|u_n-u\|_{L^\infty(0,T;H^{s})} + |_{H^s}\langle u_n(s)-u(s), h_\delta\rangle_{H^{-s}}|
\\ 
& \le 2\delta \sup_{n \in \mathbb{N}}\|u_n\|_{L^\infty(0,T; H^{s})} + |_{H^{s}}\langle u_n(s)-u(s), h_\delta\rangle_{H^{-s}}|
\\
& \le 2\delta r + |_{H^s}\langle u_n(s)-u(s), h_\delta\rangle_{H^{-s}}|.
\end{align*}
Hence, we infer 
\begin{equation*}
\sup_{t \in [0,T]} |_{H^{s}}\langle u_n(t)-u(t), h\rangle_{H^{-s}}| \le 2\delta r + \sup_{t \in [0,T]}  |_{H^{s}}\langle u_n(t)-u(t), h_\delta\rangle_{H^{-s}}|.
\end{equation*}
Since $h_\delta \in H^{s+1}$, by Step 1, passing to the upper limit as $n \rightarrow \infty$, we obtain 
\begin{equation*}
\limsup_{n \rightarrow \infty} \sup_{t \in [0,T]} |_{H^{s}}\langle u_n(t)-u(t), h\rangle_{H^{-s}}| \le 2\delta r,
\end{equation*}
and, by the arbitrarieness of $\delta$, we infer 
\begin{equation*}
\lim_{n \rightarrow \infty} \sup_{t \in [0,T]} |_{H^{s}}\langle u_n(t)-u(t), h\rangle_{H^{-s}}| =0.
\end{equation*}
Since $C([0,T];\mathbb{B}_w)$ is a complete metric space, we infer that $u \in C([0,T];\mathbb{B}_w)$ as well. This concludes the proof.
\end{proof}

Let us now prove the following result.
\begin{proposition}
\label{tight_prop}
Let $T>0$ be any finite time. A set $K\subset C_w([0,T];H^{s})$ is relatively compact in $C_w([0,T];H^s)$ if the following conditions hold
\begin{itemize}
\item[(i)] $\displaystyle\sup_{u \in K} \|u\|_{L^\infty(0,T;H^{s})} \le r$, for some $r>0$, 
\item[(ii)] K is equicontinuous in $C([0,T];H^{-s-1})$, i.e.
\begin{equation*}
\lim_{\delta \rightarrow 0}\sup_{u \in K} \sup_{|t-s| \le \delta} \|u(t)-u(s)\|_{-s-1} =0.
\end{equation*}
\end{itemize}
\end{proposition}
\begin{proof}
Let $\{z_n\}_{n \in \mathbb{N}}\subset K$: we aim to construct a subsequence converging in $C_w([0,T];H^{s})$.

\textbf{Step 1.} Thanks to (i) we can choose a constant $C>0$ and for each $n \in \mathbb{N}$ a null set $I_n$ with $\|z_n(t)\|_{s} \le C$ for all $t \in [0,T ]\setminus  I_n$. The set $I:= \bigcup_{n \in \mathbb{N}}I_n$ is also a nullset and, for each $t \in [0,T]\setminus I$, the sequence $\{z_n(t)\}_{n \in \mathbb{N}}$ is bounded in $H^{s}$. Let $\{t_j\}_{j \in \mathbb{N}} \subset [0,T] \setminus I$ be a sequence which is dense in $[0,T]$. By construction, the embedding $H^{s} \subset H^{-s-1}$ is compact. Thus, for any $j \in \mathbb{N}$, we can extract a Cauchy subsequence in $H^{-s-1}$, still denoted by $\{z_n(t_j)\}_{n \in \mathbb{N}}$. One obtains a common Cauchy subsequence $\{z_n(t_j)\}_{n \in \mathbb{N}}$, by means of a diagonalization argument.

Let $\varepsilon>0$; from (ii) we infer the existence of a $\delta>0$ such that 
\begin{equation}
\label{star2}
\sup_{n \in \mathbb{N}} \sup_{|t-s| \le \delta} \|z_n(t)-z_n(s)\|_{-s-1} \le \frac{\varepsilon}{3}.
\end{equation}
We now choose finitely many open balls $U_\delta^1,...,U_\delta^R$ of radius $\delta$ covering the interval $[0,T]$. By density, each of these balls contains an element of the sequence $\{t_j\}_{j \in \mathbb{N}}$, say $t_{j_r} \in U_\delta ^r$ for $r \in\{1,..., R\}$. In particular, the sequence $\{z_n(t_{j_r})\}_{n \in \mathbb{N}}$ is Cauchy for any $r \in \{1,...,R\}$. Therefore, 
\begin{equation}
\label{star3}
\|z_n(t_{j_r})-z_m(t_{j_r})\|_{-s-1} \le \frac{\varepsilon}{3}, \qquad r=1,...,R,
\end{equation}
for $n, m$ chosen sufficiently large. We now fix $t \in [0,T]$ and take $r \in \{1,...,R\}$ with $|t_{j_r}-t| \le \delta$. Exploiting \eqref{star2} and \eqref{star3} we infer 
\begin{align}
\label{star4}
\|z_n(t)-z_m(t)\|_{-s-1} \le \|z_n(t)-z_n(t_{j_r})\|_{-s-1}+\|z_n(t_{j_r})-z_m(t_{j_r})\|_{-s-1}+\|z_m(t_{j_r})-z_m(t)\|_{-s-1} \le \varepsilon.
\end{align}
This means that $\{z_n\}_{n \in \mathbb{N}}$ is a Cauchy sequence in $C([0,T];H^{-s-1})$, being the estimate \eqref{star4} uniform in $t \in [0,T]$. Therefore, there exists a subsequence of $\{z_n\}_{n \in \mathbb{N}}$, still denoted by $\{z_n\}_{n \in \mathbb{N}}$, and $z \in C([0,T];H^{-s-1})$ with $z_n \rightarrow z \in C([0,T];H^{-s-1})$ as $n \rightarrow \infty$.

\textbf{Step 2.} From (i) we infer the existence of $r>0$ with $\sup_{n \in \mathbb{N}}\|z_n\|_{L^\infty([0,T];H^{s})} \le r$. Hence, by Lemma \ref{technical_lemma}, we get $z \in C([0,T];\mathbb{B}_w)$ and $z_n \rightarrow z$ in $C([0,T];\mathbb{B}_w)$. Thus, $z_n \rightarrow z$ in $C_w([0,T];H^{s})$, and this concludes the proof.
\end{proof}

Now we deal with the space $C^{0, \beta}([0,T];H^{-s-1})$ defined in \S \ref{math_ass_sec}.
\begin{corollary}
\label{tight_cor}
Let $s$ and $s'$ as in \eqref{STAR}. Let $0< \beta < 1$ and $T>0$ be any finite time. Then the embedding
\begin{equation*}
L^\infty(0,T;H^{s}) \cap C^{0, \beta}([0,T];H^{-s-1}) \hookrightarrow C([0,T];H^{s'}) \cap C_w([0,T];H^{s}) 
\end{equation*}
is compact.
\end{corollary}
\begin{proof}
The compactness of the embedding $L^\infty(0,T;H^{s}) \cap C^{0, \beta}([0,T];H^{-s-1}) \hookrightarrow C([0,T];H^{s'})$ follows from the Aubin-Lions-Simon Theorem, see \cite{AL}.

As far as the compactness of the embedding $L^\infty(0,T;H^{s}) \cap C^{0, \beta}([0,T];H^{-s-1}) \hookrightarrow  C_w([0,T];H^{s})$ is concerned, consider a set $K$ such that 
\begin{equation*}
\sup_{u \in K} \left( \|u\|_{L^\infty(0,T;H^{s})} + \|u\|_{C^{0, \beta}([0,T];H^{-s-1})} \right) \le M,
\end{equation*}
for some positive $M$. Then, it holds
\begin{equation*}
\|u(t)-u(s)\|_{-s-1} \le |t-s|^{\beta}\|u\|_{C^{0, \beta}([0,T];H^{-s-1})} \le M|t-s|^{\beta}.
\end{equation*}
This implies the equicontinuity in $C([0,T];H^{-s-1})$. We thus conclude by applying Proposition \ref{tight_prop}.
\end{proof}

\subsection{Tightness results}

\label{tight_sec}
Given $T>0$, and $s$ and $s'$ as in \eqref{STAR}, we define the locally convex space 
\begin{equation*}
Z_T:=C([0,T];H^{s'}) \cap C_w([0,T];H^s),
\end{equation*}
with the topology $\mathcal{T}$ given by the supremum of the corresponding topologies in the right-hand side.

We recall that a family of probability measures $\{\nu_n\}_{n \in \mathbb{N}}$, defined on the $\sigma$-algebra of Borel subsets of $Z_T$ is tight if for any $\eta >0$, there exists a compact subset $K_\eta$ of $Z_T$ such that 
\begin{equation*}
\sup_{n \in \mathbb{N}}\nu_n (Z_T \setminus K_\eta) \le \eta.
\end{equation*}

\begin{proposition}
\label{prop_tight}
[Tightness criterium]
Fix $0< \beta < 1$. Let $\{u_n\}_{n \in \mathbb{N}}$ be a sequence of $\mathbb{F}$-adapted $Z_T$ valued processes.
 Assume that for any $\delta>0$ there exist positive constants $R_i=R_i(\delta)$, $i=1,2$, such that 
\begin{equation}\label{tight-stima1}
\sup_{n \in \mathbb{N}}\mathbb{P}\left( \|u_n\|_{L^\infty(0,T;H^{s})}>R_1\right) \le \delta,
\end{equation}
\begin{equation}\label{tight-stima2}
\sup_{n \in \mathbb{N}}\mathbb{P}\left( \|u_n\|_{C^{0, \beta}([0,T];H^{-s-1})}>R_2\right) \le \delta.
\end{equation}
Let $\mu_n$ be the law of $u_n$ in $Z_T$. Then, the sequence $\{\mu_n\}_{n \in \mathbb{N}}$ is tight in $Z_T$.
\end{proposition}
\begin{proof}
Let 
\begin{equation*}
B_1:= \{u \in Z_T : \  \|u\|_{L^\infty(0,T;H^{s})} \le R_1\},
\end{equation*}
and 
\begin{equation*}
B_2:=\{u \in Z_T: \ \|u\|_{C^{0,\beta}([0,T]; H^{-s-1})} \le R_2\}.
\end{equation*}
Let $K$ be the closure of the set $B_1 \cap B_2$ in $Z_T$. 
By Corollary \ref{tight_cor}, $K$ is compact in $Z_T$.  
Then the tightness follows from the estimates \eqref{tight-stima1} and \eqref{tight-stima2}.
\end{proof}

\begin{remark}
\label{rem_inf}
When we work on the infinite time interval, we consider the locally convex topological spaces:
\begin{itemize}
 \item
$C([0,+\infty);H^{s'})$  with metric
$d_1(u,v)=\displaystyle\sum_{k=1}^\infty
     \frac
     1{2^k}\frac{\|u-v\|_{C([0,k];H^{s'})}}{1+\|u-v\|_{C([0,k];H^{s'})}}$;
\item
$C_w([0,+\infty);H^{s})$
 with the topology generated by the family of semi-norms\\
 $\|u\|_{k,v}=\displaystyle\sup_{0\le t\le k}|\langle u(t), v\rangle |$, $k \in
 \mathbb N, v \in H^{-s}$.
\end{itemize}
We define the  space
\begin{equation}
\label{Z_space_inf}
Z_\infty= C([0,+\infty);H^{s'}) \cap C_w([0,+\infty);H^s),
\end{equation}
which is a locally convex topological space with the topology $\mathcal T$
given by  the supremum of the corresponding topologies.
\\
The tightness in $Z_\infty$ of the laws of the processes  $u_n$ defined on the time interval 
$[0,+\infty)$ is equivalent  to the  tightness in $Z_k$
for any $k \in\mathbb N$ of the laws of the processes  $u_n$ 
defined on the time interval $[0,k]$.
\end{remark}

\section{A technical lemma}
\label{tec_lem_sec}

In this Section we prove a technical result which is needed for the proof of pathwise uniqueness of solutions in Section \ref{path_uniq_sec}.

\begin{lemma}
\label{tec_lem}
    Assume  \eqref{sigma+n} and \ref{H1}-\ref{H5}.
Given any $T>0$, if $u_1$, $u_2$ are two 
martingale solutions to \eqref{NLS_abs} on the time interval $[0,T]$, defined on the same filtered probability space $(\Omega, \mathcal{F}, \mathbb{F}, \mathbb{P})$ with respect to the same Brownian motion $W$, with the same initial data in $H^s$, then for
difference  $v:=u_1-u_2$ the equality 
\begin{equation}\label{Ito-unicita}
\begin{split}
 \|v(t)\|_{s}^2
&=2 \int_0^t\Re\big( v(r), -\im \alpha F(u_1(r))+ \im \alpha F(u_2(r))\big)_{s} {\rm d}r
\\
&+ 2\int_0^t \Re\big( v(r), \phi(u_1(r))-\phi(u_2(r))\big)_{s} {\rm d}W(r)
+\int_0^t \|\phi(u_1(r))-\phi(u_2(r))\|^2_{s}\, {\rm d}r
\end{split}
\end{equation}
is satisfied $\mathbb{P}$-a.s. for any $t \in [0,T]$.
\end{lemma}

\begin{proof}
Let $T>0$, we consider two martingale solutions on the time interval $[0,T]$,  defined on the same filtered probability space with the same Brownian motion and initial data in $H^s$; thanks to Proposition \ref{reg_Str_est} we have
$u_1,u_2 \in C([0,T];H^{s})$.
The  difference  $v:=u_1-u_2$ satisfies
\begin{equation*}
\begin{cases}
{\rm d}v(t)+\im \left[  - Av(t)+ \alpha\left(F(u_1(t))-F(u_2(t))\right)\right]{\rm d}t= \left[\phi (u_1(t))-\phi (u_2(t)) \right] {\rm d}W(t)
\\
v(0)=0.
\end{cases}
\end{equation*}
The process $v$ is not regular enough to apply the It\^o formula to $\|v\|^2_s$; thus we justify the computations leading to equality \eqref{Ito-unicita} by a regularization procedure. Given $\lambda>0$, we define $R_\lambda:=\lambda (\lambda I +A)^{-1}$. From \cite[Section 1.3]{Pazy} we know that, for any $x \in H^s$, \begin{equation}
\label{convergence R_lam}
R_\lambda x \rightarrow x \ \  \text{in} \ \ H^s \quad  \text{as} \quad  \lambda \rightarrow \infty. 
\end{equation}
Moreover, $R_\lambda$ is a linear bounded operator from $H^s$ to $H^{s+2}$. Hence, the equation
\[
R_\lambda v(t)= \im \int_0^t \left[R_\lambda Av(s) - \alpha R_\lambda \left(F(u_1(s))-F(u_2(s))\right)\right]{\rm d}s + \int_0^t R_\lambda\left[\phi (u_1(s))-\phi (u_2(s)) \right] {\rm d}W(s)
\\
\]
holds $\mathbb{P}$-a.s. in $H^s$, for any $t \in [0,T]$. We apply the It\^o formula to $\|R_\lambda v(t)\|^2_{s}$, for $t\in[0,T]$ and obtain
\begin{equation}\label{Ito-unicita_reg}
\begin{split}
\|R_\lambda v(t)\|_{s}^2
&=2\int_0^t \Re\big(R_\lambda v(r), -\im R_\lambda A(v(r))\big)_{s} {\rm d}r 
+ 2 \int_0^t \Re\big(R_\lambda v(r), -\im \alpha R_\lambda F(u_1(r))+ \im \alpha R_\lambda F(u_2(r))\big)_{s} {\rm d}r 
\\
&+ 2\int_0^t \Re\big( R_\lambda v(r), R_\lambda[\phi(u_1(r))-\phi(u_2(r))]\big)_{s} {\rm d}W(r)
+\int_0^t \|R_\lambda[\phi(u_1(r))-\phi(u_2(r))]\|^2_{s}\, {\rm d}r .
\end{split}\end{equation}
Since the operators $A$ and $R_\lambda$ commute, we have 
\begin{equation*}
\Re \big( R_\lambda v(r), iR_\lambda A v(r)\big)_{s}= \Re[\im \|A^{\frac12}R_\lambda v(r)\|^2_{s}]=0, \quad r \in [0,t], \quad \lambda >0.
\end{equation*}
We take the limit as $\lambda \rightarrow \infty$ in the terms of equation \eqref{Ito-unicita_reg}.
For $r \in [0,t]$, exploiting the regularity $u_1, u_2, v \in C([0,T];H^s)$ and $F(u_1), F(u_2) \in C([0,T];H^s)$ in virtue of \eqref{stimaF}, from \eqref{convergence R_lam}, we infer $\mathbb{P}$-a.s.
\begin{equation*}
    \Re\big(R_\lambda v(r), -\im \alpha R_\lambda F(u_1(r))+ \im \alpha R_\lambda F(u_2(r))\big)_{s} \rightarrow \Re\big( v(r), -\im \alpha  F(u_1(r))+ \im \alpha  F(u_2(r))\big)_{s}, \quad \text{as} \ \lambda \rightarrow \infty,
\end{equation*}
and the Dominated Convergence Theorem yields 
\begin{equation*}
    \int_0^t\Re\big(R_\lambda v(r), -\im \alpha R_\lambda F(u_1(r))+ \im \alpha R_\lambda F(u_2(r))\big)_{s}\, {\rm d}r \rightarrow \int_0^t \Re\big( v(r), -\im \alpha  F(u_1(r))+ \im \alpha  F(u_2(r))\big)_{s}\,{\rm d}r, \quad \text{as} \ \lambda \rightarrow \infty.
\end{equation*}
In virtue of Assumption \ref{H2}(i) and \eqref{convergence R_lam}, we get, for any $r \in [0,t]$, $\mathbb{P}$-a.s.
\begin{equation*}
   \|R_\lambda[\phi(u_1(r))-\phi(u_2(r))]\|^2_{s} \rightarrow \|\phi(u_1(r))-\phi(u_2(r))\|^2_{s}, \quad \text{as} \ \lambda \rightarrow \infty.
\end{equation*}
Thus, by the Dominated Convergence Theorem, 
\begin{equation*}
    \int_0^t \|R_\lambda[\phi(u_1(r))-\phi(u_2(r))]\|^2_{s}\, {\rm d}r \rightarrow \int_0^t  \|\phi(u_1(r))-\phi(u_2(r))\|^2_{s}{\rm d}r, \quad \text{as} \ \lambda \rightarrow \infty.
\end{equation*}
For the stochastic integral, since by Assumption \ref{H2}(i), \eqref{convergence R_lam} and the Dominated Converge Theorem, it holds 
\begin{equation*}
    \int_0^t|\Re\big( R_\lambda v(r), R_\lambda[\phi(u_1(r))-\phi(u_2(r))]\big)_{s}- \Re\big( v(r), [\phi(u_1(r))-\phi(u_2(r))]\big)_{s}|^2\, {\rm d}r \rightarrow 0, \quad \text{as} \ \lambda \rightarrow \infty,
\end{equation*}
passing to the limit along a subsequence we infer $\mathbb{P}$-a.s.
\begin{equation*}
    \int_0^t\Re\big( R_\lambda v(r), R_\lambda[\phi(u_1(r))-\phi(u_2(r))]\big)_{s} {\rm d}W(r) \rightarrow \int_0^t \Re\big(v(r), [\phi(u_1(r))-\phi(u_2(r))]\big)_{s} {\rm d}W(r), \quad \text{as} \ \lambda \rightarrow \infty.
\end{equation*}
Thus, in the limit $\lambda \rightarrow \infty$, for $t\in[0,T]$, we obtain the equality \eqref{Ito-unicita}.
\end{proof}

\section{Compact Riemannian manifolds}
\label{manifolds}
All the results we proved so far in the case of a $d$-dimensional torus hold true also in the more general case of compact Riemannian manifolds.

Let $(M, g)$ be a $d$-dimensional compact Riemannian manifold without boundary. We consider the stochastic NLS equation
\begin{equation}
\label{NLS_abs_man}
\begin{cases}
{\rm d} u(t)+\im  \left[ -A u(t)+\alpha  F(u(t))  \right] \,{\rm d}t
= \phi(u(t))  \,{\rm d} W(t) , \qquad t>0
\\
u(0)=u^0,
\end{cases}
\end{equation}
where $u:[0,\infty) \times M \times \Omega \rightarrow \mathbb{C}$, A is the negative Laplace-Beltrami operator $-\Delta_g$ on $(M,g)$ and F is given by \eqref{F}.

The notion of Sobolev spaces on Riemannian manifold is well known (see e.g. \cite[Appendix B]{BHW2019} and the therein references) and the same embeddings considered in section \ref{math_ass_sec} hold true in this framework. Moreover, the assumptions on the noise in section \ref{sec:noise} can be analougsly formulated on manifolds.

With the same notation used throughout the paper, understood now for the case of compact Riemannian manifolds, we get the following result.
\begin{theorem}
Let  $\sigma \in \mathbb N$, $s>\frac d2$.
\begin{itemize}
    \item [i)] Assume \ref{H1}-\ref{H5}. Then, for any initial datum $u^0 \in H^s$ there exists a unique global-in-time strong solution to \eqref{NLS_abs_man} with $\mathbb{P}$-a.s. paths in $C([0,\infty);H^s)$.
\item [ii)] Assume \ref{H1}-\ref{H5bis}. Then there exists at least one invariant measure $\mu$  for equation \eqref{NLS_abs_man} such that 
\begin{equation*}
\int_{H^s} \|x\|_s^p {\rm d}\mu(x)<\infty,
\end{equation*}
where $p$ is the parameter in condition \eqref{seconda-ipotesi-phi}.
\item [iii)] Assume \ref{H1}-\ref{H5bisbis}. Then 
the zero solution to \eqref{NLS_abs_man} is exponentially  stable in the $p$-mean (with $p$ the parameter in condition \eqref{terza-ipotesi-phi})  and  exponentially stable with probability one. Moreover, $\mu=\delta_0$ is the unique invariant measure for equation \eqref{NLS_abs_man}.
    \end{itemize}
\end{theorem}
\begin{proof}
The proof is a minor adaptation of the the proofs of the results in sections \ref{S-wellposedness} and \ref{erg_res_sec}. One has to be careful just in considering the Moser estimates in the case of compact Riemannian manifolds; they are  given in \cite[Corollary 2.2]{BrzezniakStrichartz}.
\end{proof}


\section*{Acknowledgements}

B.F., M.M. and M.Z. are members of Gruppo Nazionale per l’Analisi Matematica, la Probabilità e le loro Applicazioni (GNAMPA) of the Istituto Nazionale di Alta Matematica (INdAM), and gratefully acknowledge financial support through the project CUP$-$E53C22001930001. 
B.F and M.Z. than the support of the CAMRisk  at the University of Pavia.
\\
The authors are grateful to Filippo Giuliani for useful discussion.


\begin{thebibliography}{10}

\bibitem{AloMiaTan2022}
D.~Alonso-Or\'{a}n, Y.~Miao, and H.~Tang.
\newblock Global existence, blow-up and stability for a stochastic transport
  equation with non-local velocity.
\newblock {\em J. Differential Equations}, 335:244--293, 2022.

\bibitem{AppMaoRod2008}
J.~A. Appleby, X.~Mao, and A.~Rodkina.
\newblock Stabilization and destabilization of nonlinear differential equations
  by noise.
\newblock {\em IEEE Transactions on Automatic Control}, 53(3):683--691, 2008.

\bibitem{ArnCraWih1983}
L.~Arnold, H.~Crauel, and V.~Wihstutz.
\newblock Stabilization of linear systems by noise.
\newblock {\em SIAM J. Control Optim.}, 21(3):451--461, 1983.

\bibitem{AthKolMat2012}
A.~Athreya, T.~Kolba, and J.~C. Mattingly.
\newblock Propagating {L}yapunov functions to prove noise-induced
  stabilization.
\newblock {\em Electron. J. Probab.}, 17:no. 96, 38, 2012.

\bibitem{BagMauXu2023}
M.~Bagnara, M.~Maurelli, and F.~Xu.
\newblock No blow-up by nonlinear it{\^o} noise for the euler equations.
\newblock {\em Available at SSRN 4606952}, 2023.

\bibitem{BRZ}
V.~Barbu, M.~R\"{o}ckner, and D.~Zhang.
\newblock Stochastic nonlinear {S}chr\"{o}dinger equations.
\newblock {\em Nonlinear Anal.}, 136:168--194, 2016.

\bibitem{BarbuRZ}
V.~Barbu, M.~R\"{o}ckner, and D.~Zhang.
\newblock Stochastic nonlinear {S}chr\"{o}dinger equations: no blow-up in the
  non-conservative case.
\newblock {\em J. Differential Equations}, 263(11):7919--7940, 2017.

\bibitem{Benyi+Oh+Zhao}
{\'A}.~B{\'e}nyi, T.~Oh, and T.~Zhao.
\newblock Fractional {L}eibniz rule on the torus.
\newblock {\em preprint arXiv:2311.07998}, 2023.

\bibitem{Bourgain}
J.~Bourgain.
\newblock Fourier transform restriction phenomena for certain lattice subsets
  and applications to nonlinear evolution equations. {II}. {T}he
  {K}d{V}-equation.
\newblock {\em Geom. Funct. Anal.}, 3(3):209--262, 1993.

\bibitem{BFZ23}
Z.~Brze\'{z}niak, B.~Ferrario, and M.~Zanella.
\newblock Ergodic results for the stochastic nonlinear schr\"{o}dinger equation
  with large damping.
\newblock {\em J. Evol. Equ.}, 23(1):Paper No. 19, 31, 2023.

\bibitem{BFZ24}
Z.~Brze\'{z}niak, B.~Ferrario, and M.~Zanella.
\newblock Invariant measures for a stochastic nonlinear and damped 2{D}
  {S}chr\"{o}dinger equation.
\newblock {\em Nonlinearity}, 37(1):Paper No. 015001, 66, 2024.

\bibitem{BHW2019}
Z.~Brze\'{z}niak, F.~Hornung, and L.~Weis.
\newblock Martingale solutions for the stochastic nonlinear {S}chr\"{o}dinger
  equation in the energy space.
\newblock {\em Probab. Theory Related Fields}, 174(3-4):1273--1338, 2019.

\bibitem{BrzMasSei}
Z.~Brze\'{z}niak, B.~Maslowski, and J.~Seidler.
\newblock Stochastic nonlinear beam equations.
\newblock {\em Probab. Theory Related Fields}, 132(1):119--149, 2005.

\bibitem{BrzezniakStrichartz}
Z.~Brze\'{z}niak and A.~Millet.
\newblock On the stochastic {S}trichartz estimates and the stochastic nonlinear
  {S}chr\"{o}dinger equation on a compact {R}iemannian manifold.
\newblock {\em Potential Anal.}, 41(2):269--315, 2014.

\bibitem{BM11}
Z.~Brzezniak and E.~Motyl.
\newblock The existence of martingale solutions to the stochastic boussinesq
  equations.
\newblock {\em Global and Stochastic Analysis}, 1(2):175--216, 2014.

\bibitem{BvNVW}
Z.~Brze\'{z}niak, J.~M. A.~M. van Neerven, M.~C. Veraar, and L.~Weis.
\newblock It\^{o}'s formula in {UMD} {B}anach spaces and regularity of
  solutions of the {Z}akai equation.
\newblock {\em J. Differential Equations}, 245(1):30--58, 2008.

\bibitem{BGT04}
N.~Burq, P.~G\'{e}rard, and N.~Tzvetkov.
\newblock Strichartz inequalities and the nonlinear {S}chr\"{o}dinger equation
  on compact manifolds.
\newblock {\em Amer. J. Math.}, 126(3):569--605, 2004.

\bibitem{Cazenave}
T.~Cazenave.
\newblock {\em Semilinear Schrodinger Equations}, volume~10.
\newblock American Mathematical Soc., 2003.

\bibitem{CW}
T.~Cazenave and F.~B. Weissler.
\newblock The {C}auchy problem for the critical nonlinear {S}chr\"{o}dinger
  equation in {$H^s$}.
\newblock {\em Nonlinear Anal.}, 14(10):807--836, 1990.

\bibitem{Cer2005}
S.~Cerrai.
\newblock Stabilization by noise for a class of stochastic reaction-diffusion
  equations.
\newblock {\em Probab. Theory Related Fields}, 133(2):190--214, 2005.

\bibitem{ChoGes2019}
K.~Chouk and B.~Gess.
\newblock Path-by-path regularization by noise for scalar conservation laws.
\newblock {\em J. Funct. Anal.}, 277(5):1469--1498, 2019.

\bibitem{Chouk+Gubinelli_2015}
K.~Chouk and M.~Gubinelli.
\newblock Nonlinear {PDE}s with modulated dispersion {I}: {N}onlinear
  {S}chr\"{o}dinger equations.
\newblock {\em Comm. Partial Differential Equations}, 40(11):2047--2081, 2015.

\bibitem{CG19}
M.~Cirant and A.~Goffi.
\newblock On the existence and uniqueness of solutions to time-dependent
  fractional {MFG}.
\newblock {\em SIAM J. Math. Anal.}, 51(2):913--954, 2019.

\bibitem{CriLan2024}
D.~Crisan and O.~Lang.
\newblock Global solutions for stochastically controlled fluid dynamics models.
\newblock {\em arXiv preprint arXiv:2403.05923}, 2024.

\bibitem{DPZ96}
G.~Da~Prato and J.~Zabczyk.
\newblock {\em Stochastic equations in infinite dimensions}, volume 152 of {\em
  Encyclopedia of Mathematics and its Applications}.
\newblock Cambridge University Press, Cambridge, second edition, 2014.

\bibitem{DeBouard+Debussche_1999}
A.~De~Bouard and A.~Debussche.
\newblock A stochastic nonlinear {S}chr\"{o}dinger equation with multiplicative
  noise.
\newblock {\em Comm. Math. Phys.}, 205(1):161--181, 1999.

\bibitem{DeBouard+Debussche_2002}
A.~De~Bouard and A.~Debussche.
\newblock Finite-time blow-up in the additive supercritical stochastic
  nonlinear {S}chr\"{o}dinger equation: the real noise case.
\newblock In {\em The legacy of the inverse scattering transform in applied
  mathematics ({S}outh {H}adley, {MA}, 2001)}, volume 301 of {\em Contemp.
  Math.}, pages 183--194. Amer. Math. Soc., Providence, RI, 2002.

\bibitem{DeBouard+Debussche_2002-PTRF}
A.~De~Bouard and A.~Debussche.
\newblock On the effect of a noise on the solutions of the focusing
  supercritical nonlinear {S}chr\"{o}dinger equation.
\newblock {\em Probab. Theory Related Fields}, 123(1):76--96, 2002.

\bibitem{DeBouard+Debussche_2003}
A.~De~Bouard and A.~Debussche.
\newblock The stochastic nonlinear {S}chr\"{o}dinger equation in {$H^1$}.
\newblock {\em Stochastic Anal. Appl.}, 21(1):97--126, 2003.

\bibitem{DeBouard+Debussche_2005}
A.~De~Bouard and A.~Debussche.
\newblock Blow-up for the stochastic nonlinear {S}chr\"{o}dinger equation with
  multiplicative noise.
\newblock {\em Ann. Probab.}, 33(3):1078--1110, 2005.

\bibitem{DO05}
A.~Debussche and C.~Odasso.
\newblock Ergodicity for a weakly damped stochastic non-linear
  {S}chr\"{o}dinger equation.
\newblock {\em J. Evol. Equ.}, 5(3):317--356, 2005.

\bibitem{Debussche+Tsutsumi_2011}
A.~Debussche and Y.~Tsutsumi.
\newblock 1{D} quintic nonlinear {S}chr\"{o}dinger equation with white noise
  dispersion.
\newblock {\em J. Math. Pures Appl. (9)}, 96(4):363--376, 2011.

\bibitem{DubRev2022}
R.~Duboscq and A.~R\'{e}veillac.
\newblock On a stochastic {H}ardy-{L}ittlewood-{S}obolev inequality with
  application to {S}trichartz estimates for a noisy dispersion.
\newblock {\em Ann. H. Lebesgue}, 5:263--274, 2022.

\bibitem{DumGouMam2021}
S.~Dumont, O.~Goubet, and Y.~Mammeri.
\newblock Decay of solutions to one dimensional nonlinear {S}chr\"{o}dinger
  equations with white noise dispersion.
\newblock {\em Discrete Contin. Dyn. Syst. Ser. S}, 14(8):2877--2891, 2021.

\bibitem{EKZ}
I.~Ekren, I.~Kukavica, and M.~Ziane.
\newblock Existence of invariant measures for the stochastic damped
  {S}chr\"{o}dinger equation.
\newblock {\em Stoch. Partial Differ. Equ. Anal. Comput.}, 5(3):343--367, 2017.

\bibitem{FZ}
B.~Ferrario and M.~Zanella.
\newblock Stationary solutions for the nonlinear {S}chr\"odinger equation.
\newblock {\em arXiv:2305.10393}, 2023.

\bibitem{FlaGat}
F.~Flandoli and D.~Gatarek.
\newblock Martingale and stationary solutions for stochastic {N}avier-{S}tokes
  equations.
\newblock {\em Probab. Theory Related Fields}, 102(3):367--391, 1995.

\bibitem{Gar1988}
T.~C. Gard.
\newblock {\em Introduction to stochastic differential equations}, volume 114
  of {\em Monographs and Textbooks in Pure and Applied Mathematics}.
\newblock Marcel Dekker, Inc., New York, 1988.

\bibitem{GasGes2019}
P.~Gassiat and B.~Gess.
\newblock Regularization by noise for stochastic {H}amilton-{J}acobi equations.
\newblock {\em Probab. Theory Related Fields}, 173(3-4):1063--1098, 2019.

\bibitem{GasGesLioSou2024}
P.~Gassiat, B.~Gess, P.-L. Lions, and P.~E. Souganidis.
\newblock {\em Journal of Functional Analysis}, 286(4):110269, 2024.

\bibitem{GesSou2017}
B.~Gess and P.~E. Souganidis.
\newblock Long-time behavior, invariant measures, and regularizing effects for
  stochastic scalar conservation laws.
\newblock {\em Comm. Pure Appl. Math.}, 70(8):1562--1597, 2017.

\bibitem{GesSou2017b}
B.~Gess and P.~E. Souganidis.
\newblock Stochastic non-isotropic degenerate parabolic-hyperbolic equations.
\newblock {\em Stochastic Process. Appl.}, 127(9):2961--3004, 2017.

\bibitem{Kah}
R.~Z. Hasminski\u{\i}.
\newblock {\em Stochastic stability of differential equations}, volume~7 of
  {\em Monographs and Textbooks on Mechanics of Solids and Fluids, Mechanics
  and Analysis}.
\newblock Sijthoff \& Noordhoff, Alphen aan den Rijn---Germantown, Md., 1980.
\newblock Translated from the Russian by D. Louvish.

\bibitem{Herr+Tataru+Tzvetkov_2011}
S.~Herr, D.~Tataru, and N.~Tzvetkov.
\newblock Global well-posedness of the energy-critical nonlinear
  {S}chr\"{o}dinger equation with small initial data in {$H^1(\Bbb T^3)$}.
\newblock {\em Duke Math. J.}, 159(2):329--349, 2011.

\bibitem{H18}
F.~Hornung.
\newblock The stochastic nonlinear {S}chr\"{o}dinger equation in unbounded
  domains and non-compact manifolds.
\newblock {\em NoDEA Nonlinear Differential Equations Appl.}, 27(4):Paper No.
  40, 46, 2020.

\bibitem{Jak86}
A.~Jakubowski.
\newblock On the {S}korokhod topology.
\newblock {\em Ann. Inst. H. Poincar\'{e} Probab. Statist.}, 22(3):263--285,
  1986.

\bibitem{Jak98}
A.~Jakubowski.
\newblock The almost sure {S}korokhod representation for subsequences in
  nonmetric spaces.
\newblock {\em Teor. Veroyatnost. i Primenen.}, 42(1):209--216, 1997.

\bibitem{Kato}
T.~Kato.
\newblock On nonlinear {S}chr\"{o}dinger equations.
\newblock {\em Ann. Inst. H. Poincar\'{e} Phys. Th\'{e}or.}, 46(1):113--129,
  1987.

\bibitem{K2}
R.~Khasminskii.
\newblock {\em Stochastic stability of differential equations}, volume~66 of
  {\em Stochastic Modelling and Applied Probability}.
\newblock Springer, Heidelberg, second edition, 2012.
\newblock With contributions by G. N. Milstein and M. B. Nevelson.

\bibitem{Kim}
J.~U. Kim.
\newblock Invariant measures for a stochastic nonlinear {S}chr\"{o}dinger
  equation.
\newblock {\em Indiana Univ. Math. J.}, 55(2):687--717, 2006.

\bibitem{Kunze_2013_Yamada}
M.~C. Kunze.
\newblock On a class of martingale problems on {B}anach spaces.
\newblock {\em Electron. J. Probab.}, 18:No. 104, 30, 2013.

\bibitem{Merle_2022}
F.~Merle, P.~Rapha\"{e}l, I.~Rodnianski, and J.~Szeftel.
\newblock On blow up for the energy super critical defocusing nonlinear
  {S}chr\"{o}dinger equations.
\newblock {\em Invent. Math.}, 227(1):247--413, 2022.

\bibitem{OT}
T.~Ogawa and Y.~Tsutsumi.
\newblock Blow-up of solutions for the nonlinear {S}chr\"{o}dinger equation
  with quartic potential and periodic boundary condition.
\newblock In {\em Functional-analytic methods for partial differential
  equations ({T}okyo, 1989)}, volume 1450 of {\em Lecture Notes in Math.},
  pages 236--251. Springer, Berlin, 1990.

\bibitem{OO}
T.~Oh and M.~Okamoto.
\newblock On the stochastic nonlinear {S}chr\"{o}dinger equations at critical
  regularities.
\newblock {\em Stoch. Partial Differ. Equ. Anal. Comput.}, 8(4):869--894, 2020.

\bibitem{Ondrejat_2004_Uniqueness}
M.~Ondrej\'{a}t.
\newblock Uniqueness for stochastic evolution equations in {B}anach spaces.
\newblock {\em Dissertationes Math. (Rozprawy Mat.)}, 426:63, 2004.

\bibitem{On05}
M.~Ondrej\'{a}t.
\newblock Brownian representations of cylindrical local martingales, martingale
  problem and strong {M}arkov property of weak solutions of {SPDE}s in {B}anach
  spaces.
\newblock {\em Czechoslovak Math. J.}, 55(130)(4):1003--1039, 2005.

\bibitem{Pazy}
A.~Pazy.
\newblock {\em Semigroups of linear operators and applications to partial
  differential equations}, volume~44 of {\em Applied Mathematical Sciences}.
\newblock Springer-Verlag, New York, 1983.

\bibitem{RenTanWan2024}
P.~Ren, H.~Tang, and F.-Y. Wang.
\newblock Distribution-path dependent nonlinear spdes with application to
  stochastic transport type equations.
\newblock {\em Potential Analysis}, pages 1--29, 2024.

\bibitem{AL}
J.~Simon.
\newblock Compact sets in the space {$L^p(0,T;B)$}.
\newblock {\em Ann. Mat. Pura Appl. (4)}, 146:65--96, 1987.

\bibitem{Strauss}
W.~A. Strauss.
\newblock On continuity of functions with values in various {B}anach spaces.
\newblock {\em Pacific J. Math.}, 19:543--551, 1966.

\bibitem{SulSul}
C.~Sulem and P.-L. Sulem.
\newblock {\em The nonlinear {S}chr\"{o}dinger equation}, volume 139 of {\em
  Applied Mathematical Sciences}.
\newblock Springer-Verlag, New York, 1999.
\newblock Self-focusing and wave collapse.

\bibitem{Sy}
M.~Sy.
\newblock Almost sure global well-posedness for the energy supercritical
  {S}chr\"{o}dinger equations.
\newblock {\em J. Math. Pures Appl. (9)}, 154:108--145, 2021.

\bibitem{tang2022general}
H.~Tang and F.-Y. Wang.
\newblock A general framework for solving singular spdes with applications to
  fluid models driven by pseudo-differential noise.
\newblock {\em arXiv:2208.08312}, 2022.

\bibitem{Tsu2008}
Y.~Tsutsumi.
\newblock Time decay of solution for the {K}d{V} equation with multiplicative
  space-time noise.
\newblock {\em Differential Integral Equations}, 21(9-10):959--970, 2008.

\bibitem{Turitsyn_2012}
S.~K. Turitsyn, B.~G. Bale, and M.~P. Fedoruk.
\newblock Dispersion-managed solitons in fibre systems and lasers.
\newblock {\em Physics reports}, 521(4):135--203, 2012.

\bibitem{Z}
D.~Zhang.
\newblock Stochastic nonlinear {S}chr\"{o}dinger equations in the defocusing
  mass and energy critical cases.
\newblock {\em Ann. Appl. Probab.}, 33(5):3652--3705, 2023.

\end{thebibliography}
\end{document}